\documentclass[a4paper,11pt]{amsart}

\usepackage[colorlinks,linkcolor=blue,citecolor=blue]{hyperref}
\usepackage{latexsym, amssymb, amsmath, amsthm, mathrsfs, bbm}
\usepackage[all]{xy}
\usepackage{pstricks}
\usepackage{ifthen}

\usepackage{anysize}\marginsize{22mm}{22mm}{30mm}{30mm}
\allowdisplaybreaks[4]

\newcommand{\Dchaintwo}[4]{
\rule[-3\unitlength]{0pt}{8\unitlength}
\begin{picture}(14,5)(0,3)
\put(1,2){\ifthenelse{\equal{#1}{l}}{\circle*{2}}{\circle{2}}}
\put(2,2){\line(1,0){10}}
\put(13,2){\ifthenelse{\equal{#1}{r}}{\circle*{2}}{\circle{2}}}
\put(1,5){\makebox[0pt]{\scriptsize #2}}
\put(7,4){\makebox[0pt]{\scriptsize #3}}
\put(13,5){\makebox[0pt]{\scriptsize #4}}
\end{picture}}
\newcommand{\Dchainthree}[6]{
\rule[-3\unitlength]{0pt}{8\unitlength}
\begin{picture}(26,5)(0,3)
\put(1,2){\ifthenelse{\equal{#1}{l}}{\circle*{2}}{\circle{2}}}
\put(2,2){\line(1,0){10}}
\put(13,2){\ifthenelse{\equal{#1}{m}}{\circle*{2}}{\circle{2}}}
\put(14,2){\line(1,0){10}}
\put(25,2){\ifthenelse{\equal{#1}{r}}{\circle*{2}}{\circle{2}}}
\put(1,5){\makebox[0pt]{\scriptsize #2}}
\put(7,4){\makebox[0pt]{\scriptsize #3}}
\put(13,5){\makebox[0pt]{\scriptsize #4}}
\put(19,4){\makebox[0pt]{\scriptsize #5}}
\put(25,5){\makebox[0pt]{\scriptsize #6}}
\end{picture}}
\newcommand{\Dtriangle}[7]{
\rule[-3\unitlength]{0pt}{12\unitlength}
\begin{picture}(18,7)(0,3)
\put(4,4){\ifthenelse{\equal{#1}{l}}{\circle*{2}}{\circle{2}}}
\put(5,4){\line(1,0){8}}
\put(14,4){\ifthenelse{\equal{#1}{r}}{\circle*{2}}{\circle{2}}}
\put(4.4472,4.8944){\line(1,2){4.1056}}
\put(9,14){\ifthenelse{\equal{#1}{t}}{\circle*{2}}{\circle{2}}}
\put(13.5528,4.8944){\line(-1,2){4.1056}}
\put(2,3){\makebox[0pt][r]{\scriptsize #2}}
\put(9,16){\makebox[0pt]{\scriptsize #3}}
\put(16,3){\makebox[0pt][l]{\scriptsize #4}}
\put(6,9){\makebox[0pt][r]{\scriptsize #5}}
\put(12.5,9){\makebox[0pt][l]{\scriptsize #6}}
\put(9,1){\makebox[0pt]{\scriptsize #7}}
\end{picture}}
\newcommand{\cDtriangle}[7]{
\rule[-10\unitlength]{0pt}{\unitlength}
\begin{picture}(18,7)(0,10)
\put(4,4){\ifthenelse{\equal{#1}{l}}{\circle*{2}}{\circle{2}}}
\put(5,4){\line(1,0){8}}
\put(14,4){\ifthenelse{\equal{#1}{r}}{\circle*{2}}{\circle{2}}}
\put(4.4472,4.8944){\line(1,2){4.1056}}
\put(9,14){\ifthenelse{\equal{#1}{t}}{\circle*{2}}{\circle{2}}}
\put(13.5528,4.8944){\line(-1,2){4.1056}}
\put(2,3){\makebox[0pt][r]{\scriptsize #2}}
\put(9,17){\makebox[0pt]{\scriptsize #3}}
\put(16,3){\makebox[0pt][l]{\scriptsize #4}}
\put(6,9){\makebox[0pt][r]{\scriptsize #5}}
\put(12.5,9){\makebox[0pt][l]{\scriptsize #6}}
\put(9,1){\makebox[0pt]{\scriptsize #7}}
\end{picture}}
\newcommand{\Dchainfour}[8]{
\rule[-3\unitlength]{0pt}{5\unitlength}
\begin{picture}(38,5)(0,3)
\put(1,2){\ifthenelse{\equal{#1}{1}}{\circle*{2}}{\circle{2}}}
\put(2,2){\line(1,0){10}}
\put(13,2){\ifthenelse{\equal{#1}{2}}{\circle*{2}}{\circle{2}}}
\put(14,2){\line(1,0){10}}
\put(25,2){\ifthenelse{\equal{#1}{3}}{\circle*{2}}{\circle{2}}}
\put(26,2){\line(1,0){10}}
\put(37,2){\ifthenelse{\equal{#1}{4}}{\circle*{2}}{\circle{2}}}
\put(1,5){\makebox[0pt]{\scriptsize #2}}
\put(7,4){\makebox[0pt]{\scriptsize #3}}
\put(13,5){\makebox[0pt]{\scriptsize #4}}
\put(19,4){\makebox[0pt]{\scriptsize #5}}
\put(25,5){\makebox[0pt]{\scriptsize #6}}
\put(31,4){\makebox[0pt]{\scriptsize #7}}
\put(37,5){\makebox[0pt]{\scriptsize #8}}
\end{picture}}

\newcommand{\bigDchainfour}[8]{
\rule[-3\unitlength]{0pt}{5\unitlength}
\begin{picture}(42,5)(0,3)
\put(1,2){\ifthenelse{\equal{#1}{1}}{\circle*{2}}{\circle{2}}}
\put(2,2){\line(1,0){14}}
\put(17,2){\ifthenelse{\equal{#1}{2}}{\circle*{2}}{\circle{2}}}
\put(18,2){\line(1,0){10}}
\put(29,2){\ifthenelse{\equal{#1}{3}}{\circle*{2}}{\circle{2}}}
\put(30,2){\line(1,0){10}}
\put(41,2){\ifthenelse{\equal{#1}{4}}{\circle*{2}}{\circle{2}}}
\put(1,5){\makebox[0pt]{\scriptsize #2}}
\put(9,4){\makebox[0pt]{\scriptsize #3}}
\put(17,5){\makebox[0pt]{\scriptsize #4}}
\put(23,4){\makebox[0pt]{\scriptsize #5}}
\put(29,5){\makebox[0pt]{\scriptsize #6}}
\put(35,4){\makebox[0pt]{\scriptsize #7}}
\put(41,5){\makebox[0pt]{\scriptsize #8}}
\end{picture}}

\newcommand{\Dthreefork}[8]{
\rule[-9\unitlength]{0pt}{12\unitlength}
\begin{picture}(28,12)(0,9)
\put(2,10){\ifthenelse{\equal{#1}{l}}{\circle*{2}}{\circle{2}}}
\put(3,10){\line(1,0){10}}
\put(14,10){\ifthenelse{\equal{#1}{m}}{\circle*{2}}{\circle{2}}}
\put(15,10){\line(1,1){7}}
\put(15,10){\line(1,-1){7}}
\put(22,18){\ifthenelse{\equal{#1}{t}}{\circle*{2}}{\circle{2}}}
\put(22,2){\ifthenelse{\equal{#1}{b}}{\circle*{2}}{\circle{2}}}
\put(2,12){\makebox[0pt]{\scriptsize #2}}
\put(8,11){\makebox[0pt]{\scriptsize #3}}
\put(14,12){\makebox[0pt]{\scriptsize #4}}
\put(19,16){\makebox[0pt][r]{\scriptsize #5}}
\put(19,4){\makebox[0pt][r]{\scriptsize #6}}
\put(24,17){\makebox[0pt][l]{\scriptsize #7}}
\put(24,2){\makebox[0pt][l]{\scriptsize #8}}
\end{picture}}

\newcommand{\Drightofway}[9]{
\rule[-9\unitlength]{0pt}{12\unitlength}
\begin{picture}(28,12)(0,9)
\put(2,10){\ifthenelse{\equal{#1}{l}}{\circle*{2}}{\circle{2}}}
\put(3,10){\line(1,0){10}}
\put(14,10){\ifthenelse{\equal{#1}{m}}{\circle*{2}}{\circle{2}}}
\put(15,10){\line(1,1){7}}
\put(15,10){\line(1,-1){7}}
\put(22,18){\ifthenelse{\equal{#1}{t}}{\circle*{2}}{\circle{2}}}
\put(22,2){\ifthenelse{\equal{#1}{b}}{\circle*{2}}{\circle{2}}}
\put(22,3){\line(0,1){14}}
\put(2,12){\makebox[0pt]{\scriptsize #2}}
\put(8,11){\makebox[0pt]{\scriptsize #3}}
\put(14,12){\makebox[0pt]{\scriptsize #4}}
\put(19,16){\makebox[0pt][r]{\scriptsize #5}}
\put(19,4){\makebox[0pt][r]{\scriptsize #6}}
\put(24,18){\makebox[0pt][l]{\scriptsize #7}}
\put(23,10){\makebox[0pt][l]{\scriptsize #8}}
\put(24,1){\makebox[0pt][l]{\scriptsize #9}}
\end{picture}}

\def \To{\longrightarrow}
\def \dim{\operatorname{dim}}
\def \gr{\operatorname{gr}}
\def \Hom{\operatorname{Hom}}

\def \ord{\operatorname{ord}}

\def \id{\operatorname{id}}

\def \1{\mathbf{1}}

\def \D{\Delta}
\def \d{\delta}

\def \e{\varepsilon}
\def \M{\mathrm{M}}
\def \MM{\mathbbm{M}}
\def \R{\mathbbm{R}}
\def \N{\mathbb{N}}
\def \S{\mathcal{S}}
\def \Z{\mathbb{Z}}

\def \deg{\operatorname{deg}}
\def \m{\mathbbm{m}}

\def \g{\mathbbm{g}}
\def \G{\mathbbm{G}}
\def \k{\mathbf{k}}
\def \B{\mathcal{B}}
\def \H{\textrm{H}}

\numberwithin{equation}{section}
\numberwithin{table}{section}
\numberwithin{equation}{section}
\newtheorem{theorem}{Theorem}[section]
\newtheorem{lemma}[theorem]{Lemma}
\newtheorem{proposition}[theorem]{Proposition}
\newtheorem{corollary}[theorem]{Corollary}
\newtheorem{definition}[theorem]{Definition}
\newtheorem{example}[theorem]{Example}
\newtheorem{remark}[theorem]{Remark}

\newtheorem{convention}[theorem]{Convention}

\begin{document}
\title{Finite quasi-quantum groups of diagonal type$^\dag$}\thanks{\tiny $^\dag$Supported by NSFC 11371186, 11431010, 11571199, and 11571329.}

\subjclass[2010]{16T05, 18D10, 16G60}
\keywords{quasi-quantum group, Nichols algebra, tensor category}

\author{Hua-Lin Huang}
\address{School of Mathematical Sciences, Huaqiao University, Quanzhou 362021, China} \email{hualin.huang@hqu.edu.cn}

\author{Gongxiang Liu}
\address{Department of Mathematics, Nanjing University, Nanjing 210093, China}
\email{gxliu@nju.edu.cn}

\author{Yuping Yang}
\address{School of Mathematics and Statistics, Southwest University, Chongqing 400715, China} \email{yupingyang@swu.edu.cn}

\author{Yu Ye}
\address{School of Mathematics, University of Science and Technology of China, Hefei 230026, China} \email{yeyu@ustc.edu.cn}

\date{}
\maketitle

\begin{abstract}
The goal of the present paper is to classify an interesting class of elementary quasi-Hopf algebras, or equivalently, finite-dimensional pointed Majid algebras. By a Tannaka-Krein type duality, this determines a big class of pointed finite tensor categories. Based on some interesting observations of normalized 3-cocycles on finite abelian groups, we elucidate an explicit connection between our objective pointed Majid algebras and finite-dimensional pointed Hopf algebras over finite abelian groups. With a help of this connection and the successful theory of diagonal Nichols algebras over abelian groups, we provide a conceptual classification of finite-dimensional graded pointed Majid algebras of diagonal type. Some efficient methods of construction are also given.
\end{abstract}


\section{Introduction}
The classification problem of finite quasi-quantum groups is motivated mainly by the theory of finite tensor categories \cite{EO}. Among which, the classification of elementary quasi-Hopf algebras, or equivalently finite-dimensional pointed Majid algebras, has been paid much attention in the last one and a half decades. Quite a few examples and classification results of such algebras, and consequently the associated pointed finite tensor categories, are thus obtained, see e.g. \cite{EG1,EG2, EG3,G,A}. In these studies, Etingof and Gelaki's novel idea of constructing genuine quasi-Hopf algebras from known pointed Hopf algebras plays a key role. This also builds a substantial connection from pointed finite tensor categories to the beautiful theory of finite-dimensional pointed Hopf algebras \cite{AS1, an}, rather than just making the latter a role model in view of the obvious similarity.

The basic idea of Etingof and Gelaki in \cite{EG1,EG2,EG3} is embedding a genuine elementary quasi-Hopf algebras into an elementary quasi-Hopf algebra, possibly up to twist equivalence. The crux of these constructions is that there is a resolution for any given 3-cocycle on a cyclic group, namely, for any 3-cocycle $\sigma$ on $\Z_n=\langle g | g^n=1 \rangle,$ the pull-back $\pi^* (\sigma)$ along the natural projection $\pi: \Z_{n^2} \to \Z_n$ is a 3-coboundary on $\Z_{n^2}.$ With this idea, the result of 3-cocycles on abelian groups of form $\Z_m \times \Z_n$ obtained in \cite{bgrc1} helps us go a step forward in constructing new finite quasi-quantum groups. In our previous work \cite{qha6}, we give a complete classification of finite-dimensional coradically graded pointed Majid algebras of rank 2. As a continuation of {\it ibid.}, this paper aims to classify diagonal finite quasi-quantum groups of arbitrary rank.

In the present paper we generalize the working spirit of \cite{qha6} to the relatively general situation. Hence we need to solve four main problems as follows. The first problem, finding a resolution for any normalized 3-cocycle, lies basically in cohomology of finite abelian groups. By extending the idea of \cite{bgrc1}, we are able to give a unified and explicit formula for a complete set of representatives of normalized 3-cocycles on any finite abelian groups. Moreover, we show that a 3-cocycle is resolvable by a finite abelian group if and only if it is abelian and we give an explicit resolution if this is indeed the case. This is also the essential case where diagonal Nichols algebras occur for twisted Yetter-Drinfeld categories. For the second problem, to give a clear description of diagonal Nichols algebras in the twisted Yetter-Drinfeld category $^\mathbbm{G}_\mathbbm{G} \mathcal{YD}^\Phi,$ we transform them  to those in the usual Yetter-Drinfeld category $^G_G \mathcal{YD}$ by a delicate manipulation, where $G$ is a finite abelian group with canonical projection $\pi: G \to \mathbb{G}$ such that $\pi^* (\Phi)$ is a 3-coboundary on $G.$ The possibility of such a transformation is guaranteed by the first step. Then by combining Heckenberger's classification of arithmetic root systems \cite{H3}, we achieve a complete classification of diagonal Nichols algebras with arithmetic root systems in $^\mathbbm{G}_\mathbbm{G} \mathcal{YD}^\Phi.$ With the transformation, we can also reduce our third problem of generation into that of Nichols algebras in the usual Yetter-Drinfeld categories of finite abelian groups. With a help of Angiono's result \cite{Ang}, we extend the useful idea in \cite{qha6} to the general situation and prove that finite-dimensional pointed Majid algebras of diagonal type are generated by group-likes and skew-primitive elements. The second and third steps together provide a complete classification of finite-dimensional graded pointed Majid algebras of diagonal type in a conceptual way. Finally we shall need to turn the conceptual classification into an operable construction, our fourth problem. For any given finite abelian group with fixed 3-cocycle and a compatible arithmetic root system, the construction is essentially a computational problem of linear congruence equations. We find two efficient ways, for most cases, to generate series of new genuine finite-dimensional pointed Majid algebras.

Here is the layout of the paper. Section 2 is devoted to some preliminary materials. In Section 3, we provide an explicit formula for normalized 3-cocycles on finite abelian groups and give resolutions of the abelian ones via finite abelian groups. In Section 4, we give a complete classification of diagonal Nichols algebras with arithmetic root system in the twisted Yetter-Drinfeld category $^G_G \mathcal{YD}^\Phi$ with $\Phi$ nontrivial. Then in Section 5 we classify in a conceptual way all the connected finite-dimensional graded pointed Majid algebras of diagonal type. Finally in Section 6 we provide some methods to construct new genuine finite-dimensional pointed Majid algebras.

Throughout the paper, $\k$ is an algebraically closed field with characteristic zero and all linear spaces are over $\k.$ A left (resp. bi-) $G$-comodule $M$, by definition, is a $G$-graded (resp. bigraded) space $M=\bigoplus_{g\in G}{^{g}M}$ (resp. $M=\bigoplus_{g,h\in G} {^{g}M^{h}}$). In general, we only deal with homogeneous elements unless stated otherwise. For convenience, if $X\in {^{g}M}$ (resp. $X\in {^{g}M^{h}}$) then we use its lowercase $x$ to denote its degree, that is $x=g$ (resp. $x=gh^{-1}$). In accordance with our previous works \cite{qha1,qha2,qha3,qha6}, we only work on pointed Majid algebras. By taking linear dual, one has the version for elementary quasi-Hopf algebras.

\section{Preliminaries}
In this section we recall some preliminary concepts, notations and facts. Clearly, there are some inevitable overlaps with the counterpart of \cite{qha6}. For the completeness and for the convenience of the reader, we remain some materials presented already in {\it loc. cit.}.

\subsection{Majid algebras.}
By definition, Majid algebras are exactly the dual of Drinfeld's quasi-Hopf algebras \cite{Dr}, and can be given as follows.
\begin{definition} A Majid algebra is a coalgebra $(\MM,\D,\e)$ equipped with a
compatible quasi-algebra structure and a quasi-antipode. Namely,
there exist two coalgebra homomorphisms $$\M: \MM \otimes \MM \To \MM, \ a
\otimes b \mapsto ab, \quad \mu: \k \To \MM,\ \lambda \mapsto \lambda
1_\MM,$$ a convolution-invertible map $\Phi: \MM^{\otimes 3} \To \k$
called associator, a coalgebra antimorphism $\S: \MM \To \MM$ and two
functionals $\alpha,\beta: \MM \To \k$ such that for all $a,b,c,d \in
\MM$ the following equalities hold:
\begin{eqnarray}
&a_1(b_1c_1)\Phi(a_2,b_2,c_2)=\Phi(a_1,b_1,c_1)(a_2b_2)c_2,\\
&1_\MM a=a=a1_\MM, \\
&\Phi(a_1,b_1,c_1d_1)\Phi(a_2b_2,c_2,d_2) =\Phi(b_1,c_1,d_1)\Phi(a_1,b_2c_2,d_2)\Phi(a_2,b_3,c_3),\\
&\Phi(a,1_\MM,b)=\e(a)\e(b). \\
&\S(a_1)\alpha(a_2)a_3=\alpha(a)1_\MM, \quad a_1\beta(a_2)\S(a_3)=\beta(a)1_\MM, \label{eq2.5} \\
&\Phi(a_1,\S(a_3),a_5)\beta(a_2)\alpha(a_4)
=\Phi^{-1}(\S(a_1),a_3,\S(a_5)) \alpha(a_2)\beta(a_4)=\e(a). \label{eq2.6}
 \end{eqnarray}
Throughout we use the Sweedler sigma notation $\D(a)=a_1 \otimes
a_2$ for the coproduct and $a_1 \otimes a_2 \otimes \cdots \otimes
a_{n+1}$ for the result of the $n$-iterated application of $\D$ on
$a.$\end{definition}

\begin{example} Let $G$ be a group and $\Phi$ a normalized $3$-cocycle on $G$. It is well known that
the group algebra $\k G$ is a Hopf algebra with $\D(g)=g\otimes g,\; \S(g)=g^{-1}$ and
$\e(g)=1$ for any $g\in G$. By extending $\Phi$ trilinearly, then $\Phi \colon (\k G)^{\otimes 3} \to \k$ becomes a
 convolution-invertible map. Define two linear functions
 $\alpha,\beta \colon \k G \to \k$ just by $\alpha(g):=\e(g)$ and $\beta(g):=\frac{1}{\Phi(g,g^{-1},g)}$
 for any $g\in G$. Then $kG$ together with these $\Phi, \ \alpha$ and $\beta$
 becomes a Majid algebra. In the following, this resulting Majid algebra is denoted by $(\k G,\Phi)$.
\end{example}

Recall that, a Majid algebra $\MM$ is said to be \emph{pointed} if the underlying
coalgebra is so. Given a pointed Majid algebra $(\MM,\D, \e,
\M, \mu, \Phi,\S,\alpha,\beta),$ let $\{\MM_n\}_{n \ge 0}$ be its
coradical filtration, and $$\gr \MM = \MM_0 \oplus \MM_1/\MM_0 \oplus \MM_2/\MM_1
\oplus \cdots$$ the corresponding coradically graded coalgebra. Then naturally
$\gr \MM$ inherits from $\MM$ a graded Majid algebra structure. The
corresponding graded associator $\gr\Phi$ satisfies
$\gr\Phi(\bar{a},\bar{b},\bar{c})=0$ for all homogeneous
$\bar{a},\bar{b},\bar{c} \in \gr \MM$ unless they all lie in $\MM_0.$
Similar condition holds for $\gr\alpha$ and $\gr\beta.$ In
particular, $\MM_0$ is a Majid subalgebra and it turns out to be the
Majid algebra $(\k G, \gr\Phi)$ for $G=G(\MM),$ the set of group-like
elements of $\MM.$ We call a pointed Majid algebra $\MM$
\emph{graded} if $\MM \cong \gr \MM$ as Majid algebras. One can also see \cite{qha1} for more details on pointed Majid algebras.

\begin{definition} Let $(\MM,\D, \e, \M, \mu, \Phi,\S,\alpha,\beta)$ be a Majid algebra. A convolution-invertible linear map
$$J:\; \MM \otimes \MM \to \k$$
is called a twisting (or gauge transformation) on $\MM$ if
$$J(h,1)=\e(h)=J(1,h)$$
for all $h\in \MM$.
\end{definition}

Given a Majid algebra $\MM$ and a twisting $J,$ then one can construct
a new Majid algebra $\MM^{J}$ as follows: $\MM^{J}=\MM$ as a coalgebra and the multiplication $``\circ"$ on $\MM^{J}$ is given by
\begin{equation}
 a\circ b:=J(a_1,b_1)a_2b_2J^{-1}(a_3,b_3)
\end{equation}
for all $a,b\in \MM$. The associator $\Phi^{J}$ and the quasi-antipode $(\S^J,\alpha^J,\beta^{J})$ are given as:
$$\Phi^J(a,b,c)=J(b_1,c_1)J(a_1,b_2c_2)\Phi(a_2,b_3,c_3)J^{-1}(a_3b_4,c_4)J^{-1}(a_4,b_5),$$
$$\S^J=\S,\;\;\;\;\alpha^J(a)=J^{-1}(\S(a_1),a_3)\alpha(a_2),\;\;\;\;\beta^J(a)=J(a_1,\S(a_3))\beta(a_2)$$
for all $a,b,c \in \MM$.

\begin{definition} Two Majid algebras $\MM_1$ and $\MM_2$ are called \emph{twist equivalent} if there is a twisting $J$ on $\MM_1$ such that $\MM_1^{J}\cong \MM_2$ as Majid algebras. Denote $\MM_1 \sim \MM_2$ if $\MM_1$ is twist equivalent to $\MM_2$. We call a Majid algebra $\MM$ \emph{genuine} if it is not twist equivalent to a Hopf algebra.
\end{definition}

\subsection{Yetter-Drinfeld modules over $(\k G,\Phi)$.}
For our purpose, we recall only the definition of Yetter-Drinfeld modules over Majid algebras of form $(\k G,\Phi)$ with $G$ an \emph{abelian} group.

Assume that $V$ is a left $\k G$-comudule with comodule structure map $\delta_{L}:\; V\to \k G\otimes V$.
Define $^{g}V:=\{v\in V|\delta_{L}(v)=g\otimes v\}$ and thus $V=\bigoplus_{g\in G}\;^{g}V.$ Here we call $g$ the degree of the elements in $^{g}V$ and denote by $\deg v=g$ for $v\in {^{g}V}.$
For the $3$-cocycle $\Phi$ on $G$ and any $g\in G$, define
\begin{equation} \widetilde{\Phi}_g:\;G\times G\to \k^{\ast}, \quad (e,f)\mapsto \frac{\Phi(g,e,f)\Phi(e,f,g)}{\Phi(e,g,f)}.
\end{equation}
Direct computation shows that
$$\widetilde{\Phi}_g\in \mathbb{Z}^{2}(G,\k^{\ast}).$$

\begin{definition}\label{d2.9} A left $\k G$-comudule $V$ is a \emph{left-left Yetter-Drinfeld module}
over $(\k G,\Phi)$ if
each $^{g}V$ is a projective
$G$-representation with respect to the $2$-cocycle $\widetilde{\Phi}_g,$ namely the $G$-action $\triangleright$ on ${^{g}V}$ satisfies
\begin{equation}
e\triangleright(f\triangleright v)=\widetilde{\Phi}_g(e,f) (ef)\triangleright v,\;\;\;\; \forall e,f\in G,\;v\in \;^{g}V.
\end{equation}
\end{definition}

The category of all left-left Yetter-Drinfeld modules is denoted by $_{ G}^{ G}\mathcal{YD}^{\Phi}$.
Similarly, one can define left-right, right-left and right-right Yetter-Drinfeld modules over $(\k G,\Phi)$.
As the familiar Hopf case, $_{ G}^{ G}\mathcal{YD}^{\Phi}$ is a braided tensor category. More precisely, for any  $M, N\in \;_{ G}^{ G}\mathcal{YD}^{\Phi},$ the structure maps of $M\otimes N$ as a left-left Yetter-Drinfeld module are given by
\begin{equation}\delta_{L}(m_{g}\otimes n_{h}):=gh \otimes m_{g} \otimes n_{h},\;\;x\triangleright (m_{g}\otimes n_{h}):=\widetilde{\Phi}_x(g,h)x\triangleright m_{g}
\otimes x\triangleright n_{h}\end{equation}
for all $x,g,h\in G$ and $m_{g} \in {^{g}M},\;n_{h}\in {^{h}N}$.
The associativity constraint $a$ and the braiding $c$ of $_{ G}^{ G}\mathcal{YD}^{\Phi}$ are given respectively by
\begin{eqnarray}
&\ a((u_{e}\otimes v_{f})\otimes w_{g}) =\Phi(e,f,g)^{-1} u_{e}\otimes (v_{f}\otimes w_{g})\\
&c(u_{e}\otimes v_{f})=e \triangleright v_{f}\otimes u_{e}
\end{eqnarray}
for all $e,f,g\in G$, $u_{e}\in {^{e}U},\  v_{f}\in {^{f}V},\  w_{g} \in {^{g}W}$ and $U,V,W\in\; _{ G}^{ G}\mathcal{YD}^{\Phi}$.

\begin{remark} A left-left Yetter-Drinfeld module $V$ over $(\k G,\Phi)$ is called \emph{diagonal} if every projective $G$-representation
$^{g}V$ is a direct sum of $1$-dimensional projective representations. In this case the union of a nonzero element of each $1$-dimensional projective representation forms a basis of $V$, which is called a \emph{canonical basis} of $V$ in the paper. We point out that not like the Hopf case, here the condition of $G$ being abelian can \emph{NOT} guarantee
that every $V$ is diagonal. It turns out that all $V \in {_{ G}^{ G}\mathcal{YD}^{\Phi}}$ are diagonal if and only if $\Phi$ is an abelian cocycle, see \cite{Ng,MN}, which is different from Eilenberg-Mac Lane abelian cocycle \cite{EM}. We will discuss these cocycles in detail in Section 3.
\end{remark}

\subsection{Bosonization for pointed Majid algebras}
The theory of bosonization in a broader context can be found in \cite{majid} in terms of braided diagrams. For our purpose, it is enough to focus on the situation of graded pointed Majid algebras. For the sake of completeness and later applications, we record in the following some explicit concepts, notations and results without proof.

In the rest of the paper, we always assume that
$$\MM=\bigoplus_{i \in \N} \MM_{i}$$ is a coradically graded connected pointed Majid algebra with unit $1.$ So $\MM_0=(\k G,\Phi)$ for some
group $G$ together with a $3$-cocycle $\Phi$ on $G$.
 Let $\pi:\; \MM\to \MM_{0}$ be
the canonical projection. Then $\MM$ is a $kG$-bicomodule naturally via
$$\delta_{L}:=(\pi\otimes \id)\D,\;\;\;\;\delta_{R}:=(\id\otimes \pi)\D.$$
Thus there is a $G$-bigrading on $\MM$, that is,
$$\MM=\bigoplus_{g,h\in G}\;^{g}\MM^{h}$$
where $^{g}\MM^{h}=\{m\in \MM|\delta_{L}(m)=g\otimes m,\;\delta_{R}(m)=m\otimes h\}$. As stated in the last paragraph of the introductoion, we only deal with
homogeneous elements with respect to this $G$-bigrading in this subsection.
For example, whenever we write $\D(X)=X_1\otimes X_2,$ all $X,X_1,X_2$ are assumed homogeneous, and for any capital $X\in\; ^{g}\MM^{h}$, we use its lowercase $x$ to denote $gh^{-1}.$

Define the coinvariant subalgebra of $\MM$ by
$$\R:=\{m\in \MM|(\id\otimes \pi)\D(m)=m\otimes 1\}.$$ Clearly $1 \in \R.$
There is a $(\k G,\Phi)$-action on $\R$ via
\begin{equation}\label{eq2.13} f\triangleright X:=\frac{\Phi(fg,f^{-1},f)}{\Phi(f,f^{-1},f)}(f\cdot X)\cdot f^{-1} \end{equation}
for all $f,g\in G$ and $X\in {^{g}\R}$. Here $\cdot$ is the multiplication in $\MM$. Then $(\R,\delta_{L},\rhd)$ is a left-left Yetter-Drinfeld module over $(\k G,\Phi)$.

Moreover, there are several natural operations on $\R$ inherited from $\MM$ as follows:
\begin{eqnarray*}&\M:\;\R\otimes \R\to \R, \;\;\;\;\;(X,Y)\mapsto XY:=X\cdot Y;\\
&u \colon \k \to \R, \quad \lambda \mapsto \lambda 1; \\
&\D_{\R}:\;\R\to \R\otimes \R,\;\;\;\;X\mapsto \Phi(x_1,x_2,x_2^{-1})X_{1}\cdot x_{2}^{-1}\otimes X_{2};\\
&\e_{\R}:\;\R\to \k,\;\;\;\;\e_{\R}:=\e|_{\R};\\
&\S_{\R}:\;\R\to \R,\;\;\;\;X\mapsto \frac{1}{\Phi(x,x^{-1},x)}x\cdot \S(X).\end{eqnarray*}
Then it is routine to verify that $(\R, \M, u, \D_{\R}, \e_{\R}, \S_{\R})$ is a Hopf algebra in $^{G}_{G}\mathcal{YD}^{\Phi}$.

Conversely, let $H$ be a Hopf algebra in $_{G}^{G}\mathcal{YD}^{\Phi}.$ Since $H$ is a left $G$-comodule, there is a $G$-grading on $H$:
 $$H=\bigoplus_{x\in G} {^{x}H}$$
 where $^{x}H=\{X\in H|\delta_{L}(X)=x\otimes X\}.$ As before, we only need to deal with $G$-homogeneous elements.
 As a convention, homogeneous elements in $H$ are denoted by capital letters, say $X, Y, Z, \dots,$ and the associated degrees are denoted by their lower cases, say $x, y, z, \dots.$

 For our purpose, we also assume  that $H$ is $\mathbb{N}$-graded with $H_0=\k$. If $X\in H_{n}$, then we say that $X$ has length $n$. Moreover, we assume that both gradings are compatible in the sense that \[ H=\bigoplus_{g\in G} {^gH} =\bigoplus_{g\in G} \bigoplus_{n\in \mathbb{N}} {^{g}H_{n}}. \] For example, the Hopf algebra $\R$ in $_{G}^{G}\mathcal{YD}^{\Phi}$ considered above satisfies these assumptions as $\R=\bigoplus_{i\in \mathbb{N}}\R_{i}$ is coradically graded. In this case, we call dim$\R_1$ the rank of $\R$ and $\MM$. For any $X\in H$, we write its comultiplication as
 $$\D_{H}(X)=X_{(1)}\otimes X_{(2)}.$$
\begin{lemma}
Keep the assumptions on $H$ as above. Define on $H\otimes \k G$ a product by
\begin{equation}
(X\otimes g)(Y\otimes h)=\frac{\Phi(xg,y,h)\Phi(x,y,g)}{\Phi(x,g,y)\Phi(xy,g,h)}X(g\triangleright Y)\otimes gh,
\end{equation}
and a coproduct  by
\begin{equation}
\D(X\otimes g)=\Phi(x_{(1)},x_{(2)},g)^{-1}(X_{(1)}\otimes x_{(2)}g)\otimes (X_{(2)}\otimes g).
\end{equation}
Then $H\otimes \k G$ becomes a graded Majid algebra with a quasi-antipode $(\S,\alpha,\beta)$ given by
\begin{eqnarray}
&\S(X\otimes g)=\frac{\Phi(g^{-1},g,g^{-1})}{\Phi(x^{-1}g^{-1},xg,g^{-1})\Phi(x,g,g^{-1})}(1\otimes x^{-1}g^{-1})(\S_H(X)\otimes 1), \\
&\alpha(1\otimes g)=1,\ \ \ \alpha(X\otimes g)=0,  \\
&\beta(1\otimes g)=\Phi(g,g^{-1},g)^{-1},\ \ \beta(X\otimes g)=0,
\end{eqnarray}
here $g,h\in G$ and $X,Y$ are homogeneous elements of length  $\geq 1.$
\end{lemma}

In the following, by $H\# \k G$ we denote the resulting Majid algebra defined on $H\otimes \k G.$

\begin{proposition}\label{p4.5} Let $\MM$ and $\R$ be as before, and $\R\#\k G$ be the Majid algebra as defined in the previous proposition. Then the map
 $$F:\;\R\#\k G \to \MM,\;\;\;\;X\otimes g\mapsto Xg$$
 is an isomorphism of Majid algebras.
\end{proposition}

\subsection{Nichols algebras in $^{ G}_{ G}\mathcal{Y}\mathcal{D}^\Phi$}
Nichols algebras can be defined by various equivalent ways, see for example \cite{as2}. Here we adopt the defining method in terms of the universal property. Roughly, Nichols algebras are the analogue of the usual symmetric algebras in more general braided tensor categories.

Let $V$ be a nonzero object in $^{ G}_{ G}\mathcal{Y}\mathcal{D}^\Phi.$ By $T_{\Phi}(V)$ we denote the tensor algebra in $^{ G}_{ G}\mathcal{Y}\mathcal{D}^\Phi$ generated freely by $V.$ It is clear that $T_{\Phi}(V)$ is isomorphic to $\bigoplus_{n \geq 0}V^{\otimes \overrightarrow{n}}$ as a linear space, where $V^{\otimes \overrightarrow{n}}$ means
$\underbrace{(\cdots((}_{n-1}V\otimes V)\otimes V)\cdots \otimes V).$ This induces a natural $\mathbb{N}$-graded structure on $T_{\Phi}(V).$ Define a comultiplication on $T_{\Phi}(V)$ by $\Delta(X)=X\otimes 1+1\otimes X, \ \forall X \in V,$ a counit by $\varepsilon(X)=0,$ and an antipode by $S(X)=-X.$ These provide a graded Hopf algebra structure on $T_{\Phi}(V)$ in the braided tensor category $^{ G}_{ G}\mathcal{Y}\mathcal{D}^\Phi.$

\begin{definition}
The Nichols algebra $\B(V)$ of $V$ is defined to be the quotient Hopf algebra $T_{\Phi}(V)/I$ in $^{ G}_{ G}\mathcal{Y}\mathcal{D}^\Phi,$ where $I$ is the unique maximal graded Hopf ideal generated by homogeneous elements of degree greater than or equal to 2.
\end{definition}

To stress that our Nichols algebras may be nonassociative in some occasions, we will call an associative Nichols algebra, e.g. $\B(V)\in {^{ G}_{ G}\mathcal{Y}\mathcal{D}}$, a usual Nichols algebra. The following definition is used widely in this paper.

\begin{definition}
Let $V\in {^{ G}_{ G}\mathcal{Y}\mathcal{D}^\Phi}$ be a Yetter-Drinfeld module of diagonal type and $\{X_i|1\leq i\leq n\}$ be a canonical basis of $V,$ and $\{g_i|1\leq i\leq n\}$ is the corresponding degrees, that is $\delta_{L}(X_i)=g_i\otimes X_i$ for $1\leq i\leq n$. Then we call the subgroup $G'=\langle g_1,\cdots,g_n\rangle$ generated by $g_1,\ldots, g_n$ the support group of $V$, which is denoted by $G_{V}$.
\end{definition}
It is obvious that the definition does not depend on the choices of canonical bases and  $$G_V=G_{\B(V)}=G_{T_{\Phi}(V)}.$$

The twisting process for Majid algebras can be transferred to Nichols algebras directly. In fact, let $(V,\triangleright, \delta_{L})\in {^{ G}_{ G}\mathcal{Y}\mathcal{D}^\Phi},$  and $J$ a $2$-cochain of $G.$ Then we can define a new action $\triangleright_J$ of $G$ over $V$ by
\begin{equation}
g\triangleright_J X=\frac{J(g,x)}{J(x,g)}g\triangleright X
\end{equation}
 for $X\in V$ and $g\in G.$  We denote $(V, \triangleright_{J},\delta_L)$ by $V^{J}$ and by definition we have $V^J\in {^{ G}_{ G}\mathcal{Y}\mathcal{D}^{\Phi\ast \partial(J)}}.$ Moreover there is a tensor equivalence $(F_J,\varphi_0,\varphi_2):\ ^{ G}_{ G}\mathcal{Y}\mathcal{D}^\Phi \to {^{ G}_{ G}\mathcal{Y}\mathcal{D}^{\Phi\ast\partial (J)}}$ which takes $V$ to $V^J$ and $$\varphi_2(U,V):(U\otimes V)^J\to U^J\otimes V^J,\ \ Y\otimes Z\mapsto J(y,z)^{-1}Y\otimes Z$$ for $Y\in U,\  Z\in V.$

 Let $\B(V)$ be a usual Nichols algebra in $^{ G}_{ G}\mathcal{Y}\mathcal{D},$ then it is clear that $\B(V)^J$ is a Hopf algebra in $^{ G}_{ G}\mathcal{Y}\mathcal{D}^{\partial J}$ with multiplication $\circ$ determined by
\begin{equation}
X\circ Y=J(x,y)XY
\end{equation} for all homogenous elements $X,Y\in \B(V),$ here $x=\deg X, y=\deg Y.$ Using the same terminology as Majid algebras or quasi-Hopf algebras, we call $\B(V)$ and $\B(V)^{J}$ are twist equivalent.
The following result is obvious, but important for our exposition.
\begin{lemma}\label{L2.12}
The twisting $\B(V)^J$ of $\B(V)$ is a Nichols algebra in $^{ G}_{ G}\mathcal{Y}\mathcal{D}^{\partial J}$ and $\B(V)^J\cong \B(V^J)$.
\end{lemma}

\subsection{Arithmetic root systems and generalized Dynkin diagrams.}\label{subsec2.5}
Arithmetic root systems are invariants of Nichols algebras of diagonal type with certain finiteness property. A complete classification of arithmetic root systems was given by Heckenberger \cite{H3}. This is a crucial ingredient for the classification program of finite-dimensional pointed Hopf algebras, and turns out to be equally important in the broader situation of pointed Majid algebras.

Suppose $\B(V)$ is a usual Nichols algebra of diagonal type in $^{ G}_{ G}\mathcal{Y}\mathcal{D}.$ Let $\{X_i|1\leq i\leq n\}$ be a canonical basis of $V$ with $\delta_l(X_i)=h_i\otimes X_i.$ The structure constants of $\B(V)$ are $\{q_{ij}|1\leq i,j\leq n\}$ such that $h_i\triangleright X_j=q_{ij}X_j.$
Let $E=\{e_i|1\leq i\leq d\}$ be a canonical basis of $\Z^n,$ and $\chi$ be a bicharacter of $\Z^n$ determined by $\chi(e_i,e_j)=q_{ij}.$ As defined in \cite[Sec.3]{H0}, $\triangle^{+}(\B(V))$ is the set of degrees of the (restricted) Poincare-Birkhoff-Witt generators counted with multiplicities and $\bigtriangleup(\B(V)):= \bigtriangleup^{+}(\B(V))\bigcup -\bigtriangleup^{+}(\B(V))$, which is called the root system of $\B(V)$. Moreover, the triple $(\bigtriangleup=\bigtriangleup(\B(V),\chi, E))$ is called an arithmetic root system of $\B(V)$ if the corresponding Weyl groupoid $W_{\chi, E}$ is full and finite, see \cite[Sec.2,3]{h2}. In this case, we denote this arithmetic root system by $\bigtriangleup(\B(V))_{\chi,E}$ for brevity.
If there is another arithmetic root system $\bigtriangleup_{\chi',E'},$ and an isomorphism $\tau:\Z^n\to \Z^n$ such that
$$\tau(E)=E',\ \ \ \ \chi'(\tau(e),\tau(e))=\chi(e,e),$$
$$\chi'(\tau(e_1),\tau(e_2))\chi'(\tau(e_2),\tau(e_1))=\chi(e_1,e_2)\chi'(e_2,e_1),$$ then we say that $\bigtriangleup_{\chi,E}$ and $\bigtriangleup_{\chi',E'}$ are twist equivalent.

A generalized Dynkin diagram is an invariant of arithmetic root systems, and it can determine arithmetic root systems up to twist equivalence.

\begin{definition}
The generalized Dynkin diagram of an arithmetic root system $\bigtriangleup_{\chi,E}$ is a nondirected graph $\mathcal{D}_{\chi,E}$ with the following properties:
\begin{itemize}
\item[1)]There is a bijective map $\phi$ from $I=\{ 1, 2, \dots, d \}$ to the set of vertices of $\mathcal{D}_{\chi,E}.$
\item[2)]For all $1\leq i\leq d,$ the vertex $\phi(i)$ is labelled  by $q_{ii}.$
\item[3)]For all $1\leq i,j\leq d,$ the number  $n_{ij}$ of edges between $\phi(i)$ and $\phi(j)$ is either $0$ or $1.$ If $i=j$ or $q_{ij}q_{ji}=1$ then $n_{ij}=0,$ otherwise $n_{ij}=1$ and the edge is labelled by $\widetilde{q_{ij}}=q_{ij}q_{ji}$ for $1\leq i<j\leq n.$
\end{itemize}
\end{definition}
An arithmetic root system is called connected provided the corresponding generalized Dynkin diagram $\mathcal{D}_{\chi,E}$ is connected. All the connected arithmetic root systems are classified and the corresponding generalized Dynkin diagrams are listed in \cite{h,H3}.

\section{Normalized $3$-cocycles over finite abelian groups}
The aim of this section is threefold: Firstly, we will give a unified formula for $3$-cocycles over a finite abelian group; Secondly, we want to develop a method to determine when a $3$-cocycle is a coboundary;  At last, we want to discuss the ``resolution" problem, i.e., for every  $3$-cocycle $\Phi$ on $\mathbb{Z}_{\mathbbm{m}_1}\times \cdots\times  \mathbb{Z}_{\mathbbm{m}_n},$ is there a bigger abelian group $G$ together with a group epimorphism $\pi \colon G \to\mathbb{Z}_{\mathbbm{m}_1}\times \cdots\times  \mathbb{Z}_{\mathbbm{m}_n}$ such that the pull-back $\pi^{*}(\Phi)$ is a
coboundary on $G$? From which, we find that there are essential differences between different $3$-cocycles and we get the definition of abelian cocycles again, which was already discussed by Ng \cite{Ng} and Mason-Ng \cite{MN}.
\subsection{A unified formula for $3$-cocycles} Let $G$ be a group and $(B_{\bullet},\partial_{\bullet})$ its bar resolution. By applying $\Hom_{\mathbbm{Z}G}(-,\k^{\ast})$
we get a complex $(B^{\ast}_{\bullet},\partial^{\ast}_{\bullet})$, where $\k^{\ast}=\k\setminus \{0\}$ is a trivial
$G$-module.

Now let $G$ be a finite abelian group. Thus $G\cong \mathbbm{Z}_{m_{1}}\times\cdots \times\mathbbm{Z}_{m_{k}}$.
For every $\mathbbm{Z}_{m_{i}}$, we fix a generator $g_{i}$  throughout this paper for $1\leq i\leq k$. It is known, see e.g. \cite[Sec. 6.2]{Wei}, that the following periodic sequence is a projective resolution for the trivial $\mathbbm{Z}_{m_{i}}$-module $\mathbbm{Z}:$
\begin{equation}\label{eq;3}\cdots\longrightarrow \mathbbm{Z}\mathbbm{Z}_{m_{i}}\stackrel{T_{i}}\longrightarrow
\mathbbm{Z}\mathbbm{Z}_{m_{i}}\stackrel{N_{i}}\longrightarrow\mathbbm{Z}\mathbbm{Z}_{m_{i}}\stackrel{T_{i}}\longrightarrow
\mathbbm{Z}\mathbbm{Z}_{m_{i}}\stackrel{N_{i}}\longrightarrow
\mathbbm{Z}\longrightarrow 0,\end{equation}
where $T_{i}=g_{i}-1$ and $N_{i}=\sum_{j=0}^{m_{i}-1}g_{i}^{j}$.

We construct the tensor product of the above periodic resolutions for
$G$. Let $K_{\bullet}$ be the following complex of projective (in fact,
free) $\mathbbm{Z}G$-modules. For
each sequence  $a_{1},\ldots,a_{k}$ of nonnegative integers, let $\Psi(a_{1},\ldots,a_{k})$ be a free
generator in degree $a_{1}+\cdots+a_{k}$. Define:
$$K_{m}:=\bigoplus_{a_{1}+\cdots+a_{k}=m} (\mathbbm{Z}G)\Psi(a_{1},\ldots,a_{k}),$$
and
$$d_{i}(\Psi(a_{1},\ldots,a_{k}))=\left \{
\begin{array}{lll} 0, &\;\;\;\;a_{i}=0,
\\ (-1)^{\sum_{l<i}a_{l}}N_{i}\Psi(a_{1},\ldots,a_{i}-1,\ldots,a_{k}), &  \ 0\neq a_{i}\ \mathrm{even},\\ (-1)^{\sum_{l<i}a_{l}}T_{i}\Psi(a_{1},\ldots,a_{i}-1,\ldots,a_{k}), & \  a_{i}\ \textrm{odd},
\end{array}\right.$$
for $1\leq i\leq k$. The differential $d$ is defined to be $d_{1}+\cdots +d_{k}$.

\begin{lemma}\label{l1.3} $(K_{\bullet},d)$ is a free resolution of the trivial $\mathbbm{Z}G$-module $\mathbbm{Z}$.
\end{lemma}
\begin{proof} Observer that $(K_{\bullet},d)$ is exactly the tensor product of the complexes \eqref{eq;3}. Thus
the lemma follows from   the K\"unneth formula for
complexes, see \cite[(3.6.3)]{Wei}.
\end{proof}

For convenience, we fix the following notations. For any $1\leq r\leq k$, define $\Psi_{r}:=\Psi(0,\ldots,1,\ldots,0)$ where $1$ lies in the $r$-th position.
For any $1\leq r\leq s\leq k$, define $\Psi_{r,s}:=\Psi(0,\ldots,1,\ldots,1,\ldots,0)$ where $1$ lies in both the $r$-th and the $s$-th position if $r<s$ and
 $\Psi_{r,r}:=\Psi(0,\ldots,2,\ldots,0)$ where $2$ lies in the $r$-th position. Similarly, one can define $\Psi_{r,s,t},\Psi_{r,s,s},\Psi_{r,r,s}$ and $\Psi_{r,r,r}$ for $1\leq r\leq k$, $1\leq r<s\leq k$ and $1\leq r<s<r\leq k$. One could  even define $\Psi_{i,j,s,t},\Psi_{i,i,j,s}$, $\Psi_{i,j,s,s}$, $\Psi_{i,j,j,s}$, $\Psi_{i,i,j,j}$, $\Psi_{i,i,i,j}$, $\Psi_{i,j,j,j}$,
 and $\Psi_{i,i,i,i}$  for $1\leq i\leq k$, $1\leq i<j\leq k$, $1\leq i<j<s\leq k$ and $1\leq i<j<s<t\leq k$ respectively.
 Now it it clear that any cochain $f\in \Hom_{\mathbbm{Z}G}(K_{3},k^{\ast})$ is uniquely determined by its values on $\Psi_{r,s,t},\Psi_{r,s,s},\Psi_{r,r,s}$
 and $\Psi_{r,r,r}$ for $1\leq r\leq k$, $1\leq r<s\leq k$ and $1\leq r<s<t\leq k$.  For such numbers, we let $f_{r,s,t}=f(\Psi_{r,s,t}),f_{r,s,s}=f(\Psi_{r,s,s}),f_{r,r,s}=f(\Psi_{r,r,s})$ and $f_{r,r,r}=f(\Psi_{r,r,r})$.

\begin{lemma}\label{l1.4} The $3$-cochain $f\in \Hom_{\mathbbm{Z}G}(K_{3},\k^{\ast})$ is a cocycle if and only if for all $1\leq r\leq k$, $1\leq r<s\leq k$ and $1\leq r<s<t\leq k$,
\begin{equation}\label{eq;1.4} f_{r,r,r}^{m_{r}}=1,\;\;f_{r,s,s}^{m_{r}}f_{r,r,s}^{m_{s}}=1,\;\;f_{r,s,t}^{m_{r}}=f_{r,s,t}^{m_{s}}=f_{r,s,t}^{m_{t}}=1.
\end{equation}
\end{lemma}
\begin{proof} The proof follows by direct computations. By definition, the cochain $f$ is a $3$-cocycle if and only if $1=d^{\ast}(f)(\Psi_{i,j,s,t})=f(d(\Psi_{i,j,s,t}))$ for all
 $1\leq i\leq j\leq s\leq t\leq k$.
For any $a\in \k^{\ast}$, it is clear that $T_{i}\cdot a=1$ since $\k^{\ast}$ is considered as a trivial $G$-module. Therefore we only need to consider the condition $1=d^{\ast}(f)(\Psi_{i,j,s,t})$
in the cases: $i=j=s=t$, $i=j<s<t$, $i<j=s<t$, $i<j<s=t$ and $i=j<s=t$ respectively. In case $i=j=s=t$, we have $1=d^{\ast}(f)(\Psi_{i,i,i,i})=f(N_{i}\Psi_{i,i,i})=N_{i}\cdot f_{i,i,i}=f_{i,i,i}^{m_{i}}$.
Similarly, if  $i=j<s<t$, we have $f_{i,s,t}^{m_{i}}=1$. If $i<j=s<t$, then we have $f_{i,j,t}^{-m_{j}}=1$. In case $i<j<s=t$, we then have $f_{i,j,s}^{m_{s}}=1$. Finally, if $i=j<s=t$, we have
$f_{i,s,s}^{m_{i}}f_{i,i,s}^{m_{s}}=1$.  Now it is easy to see  that these relations are  the same as in Equation \eqref{eq;1.4}.
\end{proof}

\begin{lemma}\label{l1.5}The $3$-cochain $f\in \Hom_{\mathbbm{Z}G}(K_{3},\k^{\ast})$ is a coboundary if and only if for all $1\leq i<j\leq k$, there are $g_{i,j}\in \k^{\ast}$ such that
\begin{equation}\label{eq3.3} f_{i,i,j}=g_{i,j}^{m_{i}},\;\;f_{i,j,j}=g_{i,j}^{-m_{j}},\ {\rm and}\ \  f_{l,l,l}=1\;\;f_{r,s,t}=1,
\end{equation} for $1\leq l\leq k$, and $1\leq r<s<t\leq k$.
\end{lemma}
\begin{proof} By definition, $f$ is a coboundary if and only if $f=d^{\ast}(g)$ for some $2$-cochain $g\in \Hom_{\mathbbm{Z}G}(K_{2},\k^{\ast})$. For any $1\leq i\leq j\leq k$,
let $g_{i,j}:=g(\Psi_{i,j})$. Since $T_{l}\cdot a=1$ for any $a\in \k^{\ast}$, we have $d^{\ast}(g)(\Psi_{r,s,t})=d^{\ast}(g)(\Psi_{l,l,l})=1$ for $1\leq r<s<t\leq k$ and $1\leq l\leq k$. Now for all $1\leq i<j\leq k$, $f_{i,i,j}=d^{\ast}(g)(\Psi_{i,i,j})=g(N_{i}\Psi_{i,j}+T_{j}\Psi_{i,i})=g_{i,j}^{m_{i}}$ and $f_{i,j,j}=d^{\ast}(g)(\Psi_{i,j,j})=g(T_{i}\Psi_{j,j}-N_{j}\Psi_{i,j})=g_{i,j}^{-m_{j}}$.
\end{proof}

For a set of natural numbers $s_{1},\ldots,s_{t},$ by $(s_{1},\ldots,s_{t})$ we denote their greatest common divisor.

\begin{proposition}
$$\H^{3}(G,\k^{\ast})\cong \prod_{i=1}^{n}\mathbb{Z}_{m_{i}}\times \prod_{1\leq i<j\leq n}^{n}\mathbb{Z}_{(m_{i},m_{j})}\times \prod_{1\leq i<j<k\leq n}^{n}\mathbb{Z}_{(m_{i},m_{j},m_{k})}.$$
\end{proposition}
\begin{proof} By Lemmas \ref{l1.4} and \ref{l1.5}, for a $3$-cocycle $f$ one can assume that $f_{l,l,l}$ is an $m_{l}$-th root of unity and $f_{i,j,k}$ is an $(m_{i},m_{j},m_{k})$-th root of unity for all $1\leq l\leq n$ and $1\leq i<j<k\leq n$. By Lemma \ref{l1.5}, one can take $g_{i,j}=f_{i,j,j}^{-\frac{1}{m_{j}}}$ and thus one can assume that $f_{i,j,j}=1$
and $g_{i,j}^{m_{j}}=1$ for all $1\leq i<j \leq n$. By $f_{i,j,j}^{m_{i}}f_{i,i,j}^{m_{j}}=1$, one has $f_{i,i,j}^{m_{j}}=1$. Therefore, $\H^{3}(G,\k^{\ast})$ must be a quotient group of $\prod_{i=1}^{n}\mathbb{Z}_{m_{i}}\times \prod_{1\leq i<j\leq n}^{n}\mathbb{Z}_{m_{j}}\times \prod_{1\leq i<j<k\leq n}^{n}\mathbb{Z}_{(m_{i},m_{j},m_{k})}.$ Using the second relation in \eqref{eq3.3}, one may even assume that $f_{i,i,j}^{m_{i}}=1$. So the proposition is proved.
\end{proof}

For any natural number $m,$ once and for all we fix $\zeta_{m}$ to be a primitive $m$-th root of unity.

\begin{corollary}\label{co3.5}
The following \[ \left\{ f\in \Hom_{\mathbb{Z}G}(K_{3},\k^{\ast})
\left|
  \begin{array}{ll}
    f_{l,l,l}=\zeta_{m_{l}}^{a_{l}},f_{i,i,j}=\zeta_{m_{j}}^{a_{ij}}, \ f_{i,j,j}=1, \
    f_{r,s,t}=\zeta_{(m_{r},m_{s},m_{t})}^{a_{rst}}\\
   \emph{for} \ 1\leq l\leq n, \ 1\leq i<j\leq n, \ 1\leq r<s<t\leq n, \ \emph{and} \\
    0 \leq a_{l}<m_{l}, \ 0\leq a_{ij}<(m_{i},m_{j}), \ 0\leq a_{rst}<(m_{r},m_{s},m_{t})
  \end{array}
\right. \right\} \]
makes a complete set of representatives of $3$-cocycles of the complex $(K_{\bullet}^*,d_{\bullet}^*).$
\end{corollary}

Next, we want to construct a chain map. We need some more notations to present the chain map. For any positive integers $s$ and $t,$ let $[\frac{s}{t}]$ denote the integer part of $\frac{s}{t}$ and let $s_t'$ denote the remainder of division of $s$ by $t.$ When there is no risk of confusion, we drop the subscript and write simply $s'.$ The following
observation is useful in later arguments.

\begin{lemma}\label{l3.6} For any three natural numbers $s,t,r$, one has
\begin{equation}\label{eq3.4}[\frac{s+t_r'}{r}]=[\frac{s+t}{r}]-[\frac{t}{r}].
\end{equation}
\end{lemma}

\begin{proof}
$ [\frac{s+t_r'}{r}]=[\frac{s+t-[\frac{t}{r}]r}{r}]=[\frac{s+t}{r}]-[\frac{t}{r}]. $
\end{proof}

Now we are ready to give a chain map, up to the third term for our purpose, from the normalized bar resolution $(B_{\bullet},\partial_{\bullet})$ to the tensor resolution $(K_{\bullet},d_{\bullet}).$ Recall that $B_{m}$ is the free $\mathbb{Z}G$-module on the set of all symbols
$[h_{1},\ldots,h_{m}]$ with $h_{i}\in G$ and $m\geq 1$. In case $m=0$, the symbol $[\; ]$ denote $1\in \mathbb{Z}G$ and the map $\partial_{0}=\epsilon:\;B_{0}\to \mathbb{Z}$ sends
$[\; ]$ to $1$.

We define the following three  morphisms of $\mathbb{Z}G$-modules:
\begin{eqnarray*}
F_{1}: &&B_{1}\To K_{1}\\
&&[g_{1}^{i_{1}}\cdots g_{n}^{i_{n}}]\mapsto
\sum_{s=1}^{n}\sum_{\alpha_{s}=0}^{i_{s}-1}g_{1}^{i_{1}}\cdots g_{s-1}^{i_{s-1}}g_{s}^{\alpha_{s}}\Psi_{s};\\
F_{2}: &&B_{2}\To K_{2}\\
&&[g_{1}^{i_{1}}\cdots g_{n}^{i_{n}},g_{1}^{j_{1}}\cdots g_{n}^{j_{n}}]\mapsto
\sum_{s=1}^{n}g_{1}^{i_{1}+j_{1}}\cdots g_{s-1}^{i_{s-1}+j_{s-1}}[\frac{i_{s}+j_{s}}{m_{s}}]\Psi_{s,s}\\
&&-\sum_{1\leq s<t\leq n}\sum_{\alpha_{s}=0}^{j_{s}-1}\sum_{\beta_{t}=0}^{i_{t}-1}g_{1}^{i_{1}}\cdots g_{t-1}^{i_{t-1}}g_{1}^{j_{1}}\cdots g_{s-1}^{j_{s-1}}g_{s}^{\alpha_{s}}g_{t}^{\beta_{t}}\Psi_{s,t};\\
F_{3}: &&B_{3}\To K_{3}\\
&&[g_{1}^{i_{1}}\cdots g_{n}^{i_{n}},g_{1}^{j_{1}}\cdots g_{n}^{j_{n}},g_{1}^{k_{1}}\cdots g_{n}^{k_{n}}]\mapsto\\
&& \sum_{r=1}^{n}[\frac{j_{r}+k_{r}}{m_{r}}]g_{1}^{j_{1}+k_{1}}\cdots g_{r-1}^{j_{r-1}+k_{r-1}}\sum_{\beta_{r}=0}^{i_{r}-1}g_{1}^{i_{1}}\cdots g_{r-1}^{i_{r-1}}g_{r}^{\beta_{r}}\Psi_{r,r,r}+\\
&&\sum_{1\leq r<t\leq n}[\frac{j_{r}+k_{r}}{m_{r}}]g_{1}^{j_{1}+k_{1}}\cdots g_{r-1}^{j_{r-1}+k_{r-1}}
\sum_{\beta_{t}=0}^{i_{t}-1}g_{1}^{i_{1}}\cdots g_{t-1}^{i_{t-1}}g_{t}^{\beta_{t}}\Psi_{r,r,t}+\\
&&\sum_{1\leq r<t\leq n}[\frac{i_{t}+j_{t}}{m_{t}}]g_{1}^{i_{1}+j_{1}}\cdots g_{t-1}^{i_{t-1}+j_{t-1}}
\sum_{\gamma_{r}=0}^{k_{r}-1}g_{1}^{k_{1}}\cdots g_{r-1}^{k_{r-1}}g_{r}^{\gamma_{r}}\Psi_{r,t,t}-\\
&&\sum_{1\leq r<s<t\leq n}\sum_{\beta_{t}=0}^{i_{t}-1}g_{1}^{i_{1}}\cdots g_{t-1}^{i_{t-1}}g_{t}^{\beta_{t}}
\sum_{\alpha_{s}=0}^{j_{s}-1}g_{1}^{j_{1}}\cdots g_{s-1}^{j_{s-1}}g_{s}^{\alpha_{s}}
\sum_{\gamma_{r}=0}^{k_{r}-1}g_{1}^{k_{1}}\cdots g_{r-1}^{k_{r-1}}g_{r}^{\gamma_{r}}\Psi_{r,s,t}
\end{eqnarray*}
for $0\leq i_{l},j_{l},k_{l}< m_{l}$ and $1\leq l\leq n$.

\begin{proposition}\label{p2.7} The following diagram is commutative

\begin{figure}[hbt]
\begin{picture}(150,50)(50,-40)
\put(0,0){\makebox(0,0){$ \cdots$}}\put(10,0){\vector(1,0){20}}\put(40,0){\makebox(0,0){$B_{3}$}}
\put(50,0){\vector(1,0){20}}\put(80,0){\makebox(0,0){$B_{2}$}}
\put(90,0){\vector(1,0){20}}\put(120,0){\makebox(0,0){$B_{1}$}}
\put(130,0){\vector(1,0){20}}\put(160,0){\makebox(0,0){$B_{0}$}}
\put(170,0){\vector(1,0){20}}\put(200,0){\makebox(0,0){$\mathbb{Z}$}}
\put(210,0){\vector(1,0){20}}\put(240,0){\makebox(0,0){$0$}}

\put(0,-40){\makebox(0,0){$ \cdots$}}\put(10,-40){\vector(1,0){20}}\put(40,-40){\makebox(0,0){$K_{3}$}}
\put(50,-40){\vector(1,0){20}}\put(80,-40){\makebox(0,0){$K_{2}$}}
\put(90,-40){\vector(1,0){20}}\put(120,-40){\makebox(0,0){$K_{1}$}}
\put(130,-40){\vector(1,0){20}}\put(160,-40){\makebox(0,0){$K_{0}$}}
\put(170,-40){\vector(1,0){20}}\put(200,-40){\makebox(0,0){$\mathbb{Z}$}}
\put(210,-40){\vector(1,0){20}}\put(240,-40){\makebox(0,0){$0$}}

\put(40,-10){\vector(0,-1){20}}
\put(80,-10){\vector(0,-1){20}}
\put(120,-10){\vector(0,-1){20}}
\put(158,-10){\line(0,-1){20}}\put(160,-10){\line(0,-1){20}}
\put(198,-10){\line(0,-1){20}}\put(200,-10){\line(0,-1){20}}

\put(60,5){\makebox(0,0){$\partial_{3}$}}
\put(100,5){\makebox(0,0){$\partial_{2}$}}
\put(140,5){\makebox(0,0){$\partial_{1}$}}

\put(60,-35){\makebox(0,0){$d$}}
\put(100,-35){\makebox(0,0){$d$}}
\put(140,-35){\makebox(0,0){$d$}}

\put(50,-20){\makebox(0,0){$F_{3}$}}
\put(90,-20){\makebox(0,0){$F_{2}$}}
\put(130,-20){\makebox(0,0){$F_{1}$}}

\end{picture}
\end{figure}
\end{proposition}

\begin{proof} The proof is by direct but very complicated computation. The essence of the proposition lies in figuring out the morphisms $F_{1},F_{2}$ and $F_{3}$ in the first place. We hope that the proof may shed some light on the construction of them. The proof is naturally divided into three parts.

\textbf{Claim 1: $dF_{1}=\partial_{1}$. } Take any generator $[g_{1}^{i_{1}}\cdots g_{n}^{i_{n}}]\in B_{1}$, then
$\partial_{1}([g_{1}^{i_{1}}\cdots g_{n}^{i_{n}}])=(g_{1}^{i_{1}}\cdots g_{n}^{i_{n}}-1)\Psi(0,\ldots,0)$.  And,
 \begin{eqnarray*}
 dF_{1}([g_{1}^{i_{1}}\cdots g_{n}^{i_{n}}])&=&d(\sum_{s=1}^{n}\sum_{\alpha_{s}=0}^{i_{s}-1}g_{1}^{i_{1}}\cdots g_{s-1}^{i_{s-1}}g_{s}^{\alpha_{s}}\Psi_{s})\\
&=&\sum_{s=1}^{n}\sum_{\alpha_{s}=0}^{i_{s}-1}g_{1}^{i_{1}}\cdots g_{s-1}^{i_{s-1}}g_{s}^{\alpha_{s}}(g_{s}-1)\Psi(0,\ldots,0)\\
&=&\sum_{s=1}^{n}g_{1}^{i_{1}}\cdots g_{s-1}^{i_{s-1}}(g_{s}^{i_{s}}-1)\Psi(0,\ldots,0)\\
&=&(g_{1}^{i_{1}}\cdots g_{n}^{i_{n}}-1)\Psi(0,\ldots,0).
\end{eqnarray*}

\textbf{Claim 2: $dF_{2}=F_{1}\partial_{2}$. } For any generator $[g_{1}^{i_{1}}\cdots g_{n}^{i_{n}},g_{1}^{j_{1}}\cdots g_{n}^{j_{n}}]$, we have
  \begin{eqnarray*}&&F_{1}\partial_{2}([g_{1}^{i_{1}}\cdots g_{n}^{i_{n}},g_{1}^{j_{1}}\cdots g_{n}^{j_{n}}])\\
  &=&F_{1}(g_{1}^{i_{1}}\cdots g_{n}^{i_{n}}[g_{1}^{j_{1}}\cdots g_{n}^{j_{n}}]-[g_{1}^{i_{1}+j_{1}}\cdots g_{n}^{i_{n}+j_{n}}]+[g_{1}^{i_{1}}\cdots g_{n}^{i_{n}}])\\
  &=&g_{1}^{i_{1}}\cdots g_{n}^{i_{n}}\sum_{s=1}^{n}\sum_{\alpha_{s}=0}^{j_{s}-1}g_{1}^{j_{1}}\cdots g_{s-1}^{j_{s-1}}g_{s}^{\alpha_{s}}\Psi_{s}\\
  &&-\sum_{s=1}^{n}\sum_{\alpha_{s}=0}^{(i_{s}+j_{s})'-1}g_{1}^{i_{1}+j_{1}}\cdots g_{s-1}^{i_{s-1}+j_{s-1}}g_{s}^{\alpha_{s}}\Psi_{s}\\
  &&+\sum_{s=1}^{n}\sum_{\alpha_{s}=0}^{i_{s}-1}g_{1}^{i_{1}}\cdots g_{s-1}^{i_{s-1}}g_{s}^{\alpha_{s}}\Psi_{s}.
   \end{eqnarray*}
   Fix any $s$, the coefficient of $\Psi_{s}$ is
   \begin{eqnarray}\label{eq3.5}&&g_{1}^{i_{1}}\cdots g_{n}^{i_{n}}\sum_{\alpha_{s}=0}^{j_{s}-1}g_{1}^{j_{1}}\cdots g_{s-1}^{j_{s-1}}g_{s}^{\alpha_{s}} \notag \\
   &&-g_{1}^{i_{1}+j_{1}}\cdots g_{s-1}^{i_{s-1}+j_{s-1}}(\sum_{\alpha_{s}=0}^{i_{s}+j_{s}-1}g_{s}^{\alpha_{s}}-[\frac{i_{s}+j_{s}}{m_{s}}]N_{s})\\\notag
   &&+\sum_{\alpha_{s}=0}^{i_{s}-1}g_{1}^{i_{1}}\cdots g_{s-1}^{i_{s-1}}g_{s}^{\alpha_{s}}.\end{eqnarray}

Now consider $dF_{2}.$ We have
   \begin{eqnarray*} &&dF_{2}([g_{1}^{i_{1}}\cdots g_{n}^{i_{n}},g_{1}^{j_{1}}\cdots g_{n}^{j_{n}}]) \\ &=&d(\sum_{s=1}^{n}g_{1}^{i_{1}+j_{1}} \cdots g_{s-1}^{i_{s-1}+j_{s-1}}[\frac{i_{s}+j_{s}}{m_{s}}]\Psi_{s,s}) \\ &&-d(\sum_{1\leq s<t\leq n}\sum_{\alpha_{s}=0}^{j_{s}-1}\sum_{\beta_{t}=0}^{i_{t}-1}g_{1}^{i_{1}}\cdots g_{t-1}^{i_{t-1}}g_{1}^{j_{1}}\cdots g_{s-1}^{j_{s-1}}g_{s}^{\alpha_{s}}g_{t}^{\beta_{t}}\Psi_{s,t}).
   \end{eqnarray*}
In this expression, the coefficient of $\Psi_{s}$ is
   \begin{eqnarray*}&&g_{1}^{i_{1}+j_{1}}\cdots g_{s-1}^{i_{s-1}+j_{s-1}}[\frac{i_{s}+j_{s}}{m_{s}}]N_{s}\\
   &&-\sum_{1\leq t<s}\sum_{\beta_{s}=0}^{i_{s}-1}g_{1}^{i_{1}}\cdots g_{s-1}^{i_{s-1}}g_{1}^{j_{1}}\cdots g_{t-1}^{j_{t-1}}(g_{t}^{j_{t}}-1)g_{s}^{\beta_{s}}\\
   &&+\sum_{s<t\leq n}\sum_{\alpha_{s}=0}^{j_{s}-1}g_{1}^{i_{1}}\cdots g_{t-1}^{i_{t-1}}g_{1}^{j_{1}}\cdots g_{s-1}^{j_{s-1}}g_{s}^{\alpha_{s}}(g_{t}^{i_{t}}-1)\\
   &=&g_{1}^{i_{1}+j_{1}}\cdots g_{s-1}^{i_{s-1}+j_{s-1}}[\frac{i_{s}+j_{s}}{m_{s}}]N_{s}\\&&
   -\sum_{\beta_{s}=0}^{i_{s}-1}g_{1}^{i_{1}}\cdots g_{s-1}^{i_{s-1}}(g_{1}^{j_{1}}\cdots g_{s-1}^{j_{s-1}}-1)g_{s}^{\beta_{s}}\\
   &&+\sum_{\alpha_{s}=0}^{j_{s}-1}(g_{1}^{i_{1}}\cdots g_{n}^{i_{n}}-g_{1}^{i_{1}}\cdots g_{s}^{i_{s}})g_{1}^{j_{1}}\cdots g_{s-1}^{j_{s-1}}g_{s}^{\alpha_{s}}\\
   &=& g_{1}^{i_{1}+j_{1}}\cdots g_{s-1}^{i_{s-1}+j_{s-1}}[\frac{i_{s}+j_{s}}{m_{s}}]N_{s}\\&&
   -\sum_{\beta_{s}=0}^{i_{s}+j_{s}-1}g_{1}^{i_{1}+j_{1}}\cdots g_{s-1}^{i_{s-1}+j_{s-1}}g_{s}^{\beta_{s}}\\
   &&+\sum_{\alpha_{s}=0}^{j_{s}-1}g_{1}^{i_{1}}\cdots g_{n}^{i_{n}}g_{1}^{j_{1}}\cdots g_{s-1}^{j_{s-1}}g_{s}^{\alpha_{s}}+\sum_{\beta_{s}=0}^{i_{s}-1}g_{1}^{i_{1}}\cdots g_{s-1}^{i_{s-1}}g_{s}^{\beta_{s}},\end{eqnarray*}
   which is clearly identical with \eqref{eq3.5}. So we have $dF_{2}=F_{1}\partial_{2}.$

   \textbf{Claim 3: $dF_{3}=F_{2}\partial_{3}$. } Similarly, for any generator $[g_{1}^{i_{1}}\cdots g_{n}^{i_{n}},g_{1}^{j_{1}}\cdots g_{n}^{j_{n}}, g_{1}^{k_{1}}\cdots g_{n}^{k_{n}}]$, we have
   \begin{eqnarray*}&&F_{2}\partial_{3}([g_{1}^{i_{1}}\cdots g_{n}^{i_{n}},g_{1}^{j_{1}}\cdots g_{n}^{j_{n}},g_{1}^{k_{1}}\cdots g_{n}^{k_{n}}])\\
   &=&F_{2}(g_{1}^{i_{1}}\cdots g_{n}^{i_{n}}[g_{1}^{j_{1}}\cdots g_{n}^{j_{n}},g_{1}^{k_{1}}\cdots g_{n}^{k_{n}}]-[g_{1}^{i_{1}+j_{1}}\cdots g_{n}^{i_{n}+j_{n}},g_{1}^{k_{1}}\cdots g_{n}^{k_{n}}])\\
   &&+F_{2}([g_{1}^{i_{1}}\cdots g_{n}^{i_{n}},g_{1}^{j_{1}+k_{1}}\cdots g_{n}^{j_{n}+k_{n}}]-[g_{1}^{i_{1}}\cdots g_{n}^{i_{n}},g_{1}^{j_{1}}\cdots g_{n}^{j_{n}}])\\
   &=&g_{1}^{i_{1}}\cdots g_{n}^{i_{n}}\sum_{s=1}^{n}g_{1}^{j_{1}+k_{1}}\cdots g_{s-1}^{j_{s-1}+k_{s-1}}[\frac{j_{s}+k_{s}}{m_{s}}]\Psi_{s,s}\\
   &&-g_{1}^{i_{1}}\cdots g_{n}^{i_{n}}\sum_{1\leq s<t\leq n}\sum_{\alpha_{s}=0}^{k_{s}-1}\sum_{\beta_{t}=0}^{j_{t}-1}g_{1}^{j_{1}}\cdots g_{t-1}^{j_{t-1}}g_{1}^{k_{1}}\cdots g_{s-1}^{k_{s-1}}g_{s}^{\alpha_{s}}g_{t}^{\beta_{t}}\Psi_{s,t}\\
   &&-\sum_{s=1}^{n}g_{1}^{i_{1}+j_{1}+k_{1}}\cdots g_{s-1}^{i_{s-1}+j_{s-1}+k_{s-1}}[\frac{(i_{s}+j_{s})'+k_{s}}{m_{s}}]\Psi_{s,s}\\
   &&+\sum_{1\leq s<t\leq n}\sum_{\alpha_{s}=0}^{k_{s}-1}\sum_{\beta_{t}=0}^{(i_{t}+j_{t})'-1}g_{1}^{i_{1}+j_{1}}\cdots g_{t-1}^{i_{t-1}+j_{t-1}}g_{1}^{k_{1}}\cdots g_{s-1}^{k_{s-1}}g_{s}^{\alpha_{s}}g_{t}^{\beta_{t}}\Psi_{s,t}\\
   &&+\sum_{s=1}^{n}g_{1}^{i_{1}+j_{1}+k_{1}}\cdots g_{s-1}^{i_{s-1}+j_{s-1}+k_{s-1}}[\frac{i_{s}+(j_{s}+k_{s})'}{m_{s}}]\Psi_{s,s}\\
   &&-\sum_{1\leq s<t\leq n}\sum_{\alpha_{s}=0}^{(j_{s}+k_{s})'-1}\sum_{\beta_{t}=0}^{i_{t}-1}g_{1}^{i_{1}}\cdots g_{t-1}^{i_{t-1}}g_{1}^{j_{1}+k_{1}}\cdots g_{s-1}^{j_{s-1}+k_{s-1}}g_{s}^{\alpha_{s}}g_{t}^{\beta_{t}}\Psi_{s,t}\\
   &&-\sum_{s=1}^{n}g_{1}^{i_{1}+j_{1}}\cdots g_{s-1}^{i_{s-1}+j_{s-1}}[\frac{i_{s}+j_{s}}{m_{s}}]\Psi_{s,s}\\
   &&+\sum_{1\leq s<t\leq n}\sum_{\alpha_{s}=0}^{j_{s}-1}\sum_{\beta_{t}=0}^{i_{t}-1}g_{1}^{i_{1}}\cdots g_{t-1}^{i_{t-1}}g_{1}^{j_{1}}\cdots g_{s-1}^{j_{s-1}}g_{s}^{\alpha_{s}}g_{t}^{\beta_{t}}\Psi_{s,t}.
   \end{eqnarray*}
Note that in $(i_s+j_s)'$ we drop the subscript $m_s.$ In the previous expression, for any $1\leq s\leq n$, the coefficient of $\Psi_{s,s}$ is
    \begin{eqnarray}\label{eq3.6}&&g_{1}^{i_{1}}\cdots g_{n}^{i_{n}}g_{1}^{j_{1}+k_{1}}\cdots g_{s-1}^{j_{s-1}+k_{s-1}}[\frac{j_{s}+k_{s}}{m_{s}}] \notag \\
    &&+g_{1}^{i_{1}+j_{1}+k_{1}}\cdots g_{s-1}^{i_{s-1}+j_{s-1}+k_{s-1}}([\frac{i_{s}+j_{s}}{m_{s}}]-[\frac{j_{s}+k_{s}}{m_{s}}])\\\notag
    &&-g_{1}^{i_{1}+j_{1}}\cdots g_{s-1}^{i_{s-1}+j_{s-1}}[\frac{i_{s}+j_{s}}{m_{s}}],\end{eqnarray} where Lemma \ref{l3.6} is applied.
    For any $1\leq s< t\leq n$, the coefficient of $\Psi_{s,t}$ is
     \begin{eqnarray}\label{eq3.7}&&-g_{1}^{i_{1}}\cdots g_{n}^{i_{n}}\sum_{\alpha_{s}=0}^{k_{s}-1}\sum_{\beta_{t}=0}^{j_{t}-1}g_{1}^{j_{1}}\cdots g_{t-1}^{j_{t-1}}g_{1}^{k_{1}}\cdots g_{s-1}^{k_{s-1}}g_{s}^{\alpha_{s}}g_{t}^{\beta_{t}}\notag \\
    &&+\sum_{\alpha_{s}=0}^{k_{s}-1}g_{1}^{i_{1}+j_{1}}\cdots g_{t-1}^{i_{t-1}+j_{t-1}}g_{1}^{k_{1}}\cdots g_{s-1}^{k_{s-1}}g_{s}^{\alpha_{s}}(\sum_{\beta_{t}=0}^{i_{t}+j_{t}-1}g_{t}^{\beta_{t}}
    -[\frac{i_{t}+j_{t}}{m_{t}}]N_{t})\\\notag
    &&-\sum_{\beta_{t}=0}^{i_{t}-1}g_{1}^{i_{1}}\cdots g_{t-1}^{i_{t-1}}g_{t}^{\beta_{t}}g_{1}^{j_{1}+k_{1}}\cdots g_{s-1}^{j_{s-1}+k_{s-1}}(\sum_{\alpha_{s}=0}^{j_{s}+k_{s}-1}g_{s}^{\alpha_{s}}-
    [\frac{j_{s}+k_{s}}{m_{s}}]N_{s})\\\notag
    &&+\sum_{\alpha_{s}=0}^{j_{s}-1}\sum_{\beta_{t}=0}^{i_{t}-1}g_{1}^{i_{1}}\cdots g_{t-1}^{i_{t-1}}g_{1}^{j_{1}}\cdots g_{s-1}^{j_{s-1}}g_{s}^{\alpha_{s}}g_{t}^{\beta_{t}}.\end{eqnarray}

For $dF_{3},$ we have
      \begin{eqnarray*}&&dF_{3}([g_{1}^{i_{1}}\cdots g_{n}^{i_{n}},g_{1}^{j_{1}}\cdots g_{n}^{j_{n}},g_{1}^{k_{1}}\cdots g_{n}^{k_{n}}])\\
      &=& d(\sum_{r=1}^{n}[\frac{j_{r}+k_{r}}{m_{r}}]g_{1}^{j_{1}+k_{1}}\cdots g_{r-1}^{j_{r-1}+k_{r-1}}\sum_{\beta_{r}=0}^{i_{r}-1}g_{1}^{i_{1}}\cdots g_{r-1}^{i_{r-1}}g_{r}^{\beta_{r}}\Psi_{r,r,r})+\\
      &&d(\sum_{1\leq r<t\leq n}[\frac{j_{r}+k_{r}}{m_{r}}]g_{1}^{j_{1}+k_{1}}\cdots g_{r-1}^{j_{r-1}+k_{r-1}}
      \sum_{\beta_{t}=0}^{i_{t}-1}g_{1}^{i_{1}}\cdots g_{t-1}^{i_{t-1}}g_{t}^{\beta_{t}}\Psi_{r,r,t})+\\
      &&d(\sum_{1\leq r<t\leq n}[\frac{i_{t}+j_{t}}{m_{t}}]g_{1}^{i_{1}+j_{1}}\cdots g_{t-1}^{i_{t-1}+j_{t-1}}
      \sum_{\gamma_{r}=0}^{k_{r}-1}g_{1}^{k_{1}}\cdots g_{r-1}^{k_{r-1}}g_{r}^{\gamma_{r}}\Psi_{r,t,t})-\\
      &&d(\sum_{1\leq r<s<t\leq n}\sum_{\beta_{t}=0}^{i_{t}-1}g_{1}^{i_{1}}\cdots g_{t-1}^{i_{t-1}}g_{t}^{\beta_{t}}
      \sum_{\alpha_{s}=0}^{j_{s}-1}g_{1}^{j_{1}}\cdots g_{s-1}^{j_{s-1}}g_{s}^{\alpha_{s}}
     \sum_{\gamma_{r}=0}^{k_{r}-1}g_{1}^{k_{1}}\cdots g_{r-1}^{k_{r-1}}g_{r}^{\gamma_{r}}\Psi_{r,s,t}).\\
     \end{eqnarray*}
Note that the coefficient of $\Psi_{s,s}$ is
      \begin{eqnarray*}&&[\frac{j_{s}+k_{s}}{m_{s}}]g_{1}^{j_{1}+k_{1}}\cdots g_{s-1}^{j_{s-1}+k_{s-1}}g_{1}^{i_{1}}\cdots g_{s-1}^{i_{s-1}}(g_{s}^{i_{s}}-1)\\
      &&+\sum_{ s<t\leq n}[\frac{j_{s}+k_{s}}{m_{s}}]g_{1}^{j_{1}+k_{1}}\cdots g_{r-1}^{j_{s-1}+k_{s-1}}
      g_{1}^{i_{1}}\cdots g_{t-1}^{i_{t-1}}(g_{t}^{i_{t}}-1)\\
      &&+\sum_{1\leq r<s}[\frac{i_{s}+j_{s}}{m_{s}}]g_{1}^{i_{1}+j_{1}}\cdots g_{s-1}^{i_{s-1}+j_{s-1}}
      g_{1}^{k_{1}}\cdots g_{r-1}^{k_{r-1}}(g_{r}^{k_{r}}-1)\\
      &=&[\frac{j_{s}+k_{s}}{m_{s}}]g_{1}^{j_{1}+k_{1}}\cdots g_{s-1}^{j_{s-1}+k_{s-1}}g_{1}^{i_{1}}\cdots g_{s-1}^{i_{s-1}}(g_{s}^{i_{s}}-1)\\
      &&+[\frac{j_{s}+k_{s}}{m_{s}}]g_{1}^{j_{1}+k_{1}}\cdots g_{r-1}^{j_{s-1}+k_{s-1}}
      (g_{1}^{i_{1}}\cdots g_{n}^{i_{n}}-g_{1}^{i_{1}}\cdots g_{s}^{i_{s}})\\
      &&+[\frac{i_{s}+j_{s}}{m_{s}}]g_{1}^{i_{1}+j_{1}}\cdots g_{s-1}^{i_{s-1}+j_{s-1}}
      (g_{1}^{k_{1}}\cdots g_{s-1}^{k_{s-1}}-1),
      \end{eqnarray*} which clearly is equal to \eqref{eq3.6}.

Finally we consider the coefficient of $\Psi_{s,t}$ for $1\leq s<t\leq n$, which is
    \begin{eqnarray*}
     &&[\frac{j_{s}+k_{s}}{m_{s}}]N_{s}g_{1}^{j_{1}+k_{1}}\cdots g_{s-1}^{j_{s-1}+k_{s-1}}
      \sum_{\beta_{t}=0}^{i_{t}-1}g_{1}^{i_{1}}\cdots g_{t-1}^{i_{t-1}}g_{t}^{\beta_{t}}\\
      &&+[\frac{i_{t}+j_{t}}{m_{t}}]N_{t}g_{1}^{i_{1}+j_{1}}\cdots g_{t-1}^{i_{t-1}+j_{t-1}}
      \sum_{\gamma_{s}=0}^{k_{s}-1}g_{1}^{k_{1}}\cdots g_{s-1}^{k_{s-1}}g_{s}^{\gamma_{s}}\\
      &&-\sum_{1\leq r<s<t}\sum_{\beta_{t}=0}^{i_{t}-1}g_{1}^{i_{1}}\cdots g_{t-1}^{i_{t-1}}g_{t}^{\beta_{t}}
      \sum_{\alpha_{s}=0}^{j_{s}-1}g_{1}^{j_{1}}\cdots g_{s-1}^{j_{s-1}}g_{s}^{\alpha_{s}}
     g_{1}^{k_{1}}\cdots g_{r-1}^{k_{r-1}}(g_{r}^{k_{r}}-1)\\
     &&+\sum_{ s<r<t}\sum_{\beta_{t}=0}^{i_{t}-1}g_{1}^{i_{1}}\cdots g_{t-1}^{i_{t-1}}g_{t}^{\beta_{t}}
     g_{1}^{j_{1}}\cdots g_{r-1}^{j_{r-1}}(g_{r}^{j_{r}}-1)
     \sum_{\gamma_{s}=0}^{k_{s}-1}g_{1}^{k_{1}}\cdots g_{s-1}^{k_{s-1}}g_{s}^{\gamma_{s}}\\
     &&-\sum_{s<t<r}g_{1}^{i_{1}}\cdots g_{r-1}^{i_{r-1}}(g_{r}^{i_{r}}-1)
      \sum_{\alpha_{t}=0}^{j_{t}-1}g_{1}^{j_{1}}\cdots g_{t-1}^{j_{t-1}}g_{t}^{\alpha_{t}}
      \sum_{\gamma_{s}=0}^{k_{s}-1}g_{1}^{k_{1}}\cdots g_{s-1}^{k_{s-1}}g_{s}^{\gamma_{s}}\\
      &=&[\frac{j_{s}+k_{s}}{m_{s}}]N_{s}g_{1}^{j_{1}+k_{1}}\cdots g_{s-1}^{j_{s-1}+k_{s-1}}
      \sum_{\beta_{t}=0}^{i_{t}-1}g_{1}^{i_{1}}\cdots g_{t-1}^{i_{t-1}}g_{t}^{\beta_{t}}\\
      &&+[\frac{i_{t}+j_{t}}{m_{t}}]N_{t}g_{1}^{i_{1}+j_{1}}\cdots g_{t-1}^{i_{t-1}+j_{t-1}}
      \sum_{\gamma_{s}=0}^{k_{s}-1}g_{1}^{k_{1}}\cdots g_{s-1}^{k_{s-1}}g_{s}^{\gamma_{s}}\\
      &&-\sum_{\beta_{t}=0}^{i_{t}-1}g_{1}^{i_{1}}\cdots g_{t-1}^{i_{t-1}}g_{t}^{\beta_{t}}
      \sum_{\alpha_{s}=0}^{j_{s}-1}g_{1}^{j_{1}}\cdots g_{s-1}^{j_{s-1}}g_{s}^{\alpha_{s}}
     (g_{1}^{k_{1}}\cdots g_{s-1}^{k_{s-1}}-1)\\
     &&+\sum_{\beta_{t}=0}^{i_{t}-1}g_{1}^{i_{1}}\cdots g_{t-1}^{i_{t-1}}g_{t}^{\beta_{t}}
     (g_{1}^{j_{1}}\cdots g_{t-1}^{j_{t-1}}-g_{1}^{j_{1}}\cdots g_{s}^{j_{s}})
     \sum_{\gamma_{s}=0}^{k_{s}-1}g_{1}^{k_{1}}\cdots g_{s-1}^{k_{s-1}}g_{s}^{\gamma_{s}}\\
     &&-(g_{1}^{i_{1}}\cdots g_{n}^{i_{n}}-g_{1}^{i_{1}}\cdots g_{t}^{i_{t}})
      \sum_{\alpha_{t}=0}^{j_{t}-1}g_{1}^{j_{1}}\cdots g_{t-1}^{j_{t-1}}g_{t}^{\alpha_{t}}
      \sum_{\gamma_{s}=0}^{k_{s}-1}g_{1}^{k_{1}}\cdots g_{s-1}^{k_{s-1}}g_{s}^{\gamma_{s}}.
    \end{eqnarray*} It is not hard to see that this is equal to \eqref{eq3.7}. Therefore, $dF_{3}=F_{2}\partial_{3}.$

The proof is completed.
\end{proof}

Now we are able to accomplish the main task with a help of the results obtained above. Define $A$ to be the set of all sequences like
\begin{equation}\label{3.1}(a_{1},\ldots,a_{l},\ldots,a_{n},a_{12},\ldots,a_{ij},\ldots,a_{n-1,n},a_{123},
\ldots,a_{rst},\ldots,a_{n-2,n-1,n})\end{equation}
such that $ 0\leq a_{l}<m_{l}, \ 0\leq a_{ij}<(m_{i},m_{j}), \ 0\leq a_{rst}<(m_{r},m_{s},m_{t})$ for $1\leq l\leq n, \ 1\leq i<j\leq n, \ 1\leq r<s<t\leq n$ where $a_{ij}$ and $a_{rst}$ are ordered by the lexicographic order. In the following, the sequence \eqref{3.1} is denoted by $\underline{\mathbf{a}}$ for short.

For any $\underline{\mathbf{a}}\in A$, define a $\mathbb{Z}G$-module morphism:
\begin{eqnarray}\label{3cocycle}
&& \Phi_{\underline{\mathbf{a}}}:\;B_{3}\To \k^{\ast} \notag \\
&&[g_{1}^{i_{1}}\cdots g_{n}^{i_{n}},g_{1}^{j_{1}}\cdots g_{n}^{j_{n}},g_{1}^{k_{1}}\cdots g_{n}^{k_{n}}] \mapsto \\ &&\prod_{l=1}^{n}\zeta_{m_{l}}^{a_{l}i_{l}[\frac{j_{l}+k_{l}}{m_{l}}]}
\prod_{1\leq s<t\leq n}\zeta_{m_{t}}^{a_{st}i_{t}[\frac{j_{s}+k_{s}}{m_{s}}]}
\prod_{1\leq r<s<t\leq n}\zeta_{(m_{r},m_{s},m_{t})}^{-a_{rst}k_{r}j_{s}i_{t}}. \notag
\end{eqnarray}

\begin{proposition}\label{p3.8}  Suppose that $\k$ is an algebraically closed field of characteristic zero and $G=\mathbb{Z}_{m_{1}}\times\cdots \times\mathbb{Z}_{m_{n}}.$ Then $\{\Phi_{\underline{\mathbf{a}}}|\underline{\mathbf{a}}\in A\}$ is a complete set of representatives of normalized $3$-cocycles on $G.$
\end{proposition}
\begin{proof} This is a direct consequence of Corollary \ref{co3.5} and the definition of the map $F_{3}$ given in Proposition \ref{p2.7}.
\end{proof}

\subsection{3-coboundary} Later on, we will encounter the following problem: Given a $3$-cocycle of the complex $(B^{\ast}_{\bullet},\partial^{\ast}_{\bullet})$,
we have to determine whether it is a $3$-coboundary or not. In this subsection, we want to solve this problem in case $G$ is a finite abelian group. In fact, Lemma \ref{l1.5} already provides us an easy way. For the bar resolution, it is sufficient to give a chain map from $(K_{\bullet},d_{\bullet})$
to $(B_{\bullet},\partial_{\bullet})$, which is a kind of inverse of the chain map defined in the previous subsection and thus becomes much simpler. We use the following three morphisms of $\mathbbm{Z}G$-modules defined in \cite[Section 2]{LOZ}:
\begin{eqnarray*}
F_{1}: &&K_{1}\To B_{1},\;\;\;\;\Psi_r\mapsto [g_r];\\
F_{2}: &&K_{2}\To B_{2},\\
&&\Psi_{r,s}\mapsto [g_r,g_s]-[g_s,g_r],\\
&&\Psi_{r,r}\mapsto \sum_{l=0}^{m_{r}-1}[g_{r}^{l},g_r];\\
F_{3}: &&K_{3}\To B_{3},\\
&&\Psi_{r,s,t}\mapsto[g_r,g_s,g_t]-[g_s,g_r,g_t]-[g_r,g_t,g_s],\\
&&\;\;\;\;\;\;\;\;\;\;\;\;\;\;[g_t,g_r,g_s]+[g_s,g_t,g_r]-[g_t,g_s,g_r],\\
&&\Psi_{r,r,s}\mapsto \sum_{l=0}^{m_{r}-1}([g_{r}^{l},g_r,g_s]-[g_{r}^{l},g_s,g_r]+[g_s,g_{r}^{l},g_r]),\\
&&\Psi_{r,s,s}\mapsto \sum_{l=0}^{m_{s}-1}([g_{r},g_{s}^{l},g_s]-[g_{s}^{l},g_r,g_s]+[g_{s}^{l},g_{s},g_r]),\\
&&\Psi_{r,r,r}\mapsto \sum_{l=0}^{m_{r}-1}[g_{r},g_{r}^{l},g_r],
\end{eqnarray*}
for $0\leq r\leq k$, $0\leq r<s\leq k$ and $0\leq r<s<t\leq k$.

\begin{lemma}\label{l1.6} The following diagram is commutative:

\begin{figure}[hbt]
\begin{picture}(150,50)(50,-40)
\put(0,0){\makebox(0,0){$ \cdots$}}\put(10,0){\vector(1,0){20}}\put(40,0){\makebox(0,0){$K_{3}$}}
\put(50,0){\vector(1,0){20}}\put(80,0){\makebox(0,0){$K_{2}$}}
\put(90,0){\vector(1,0){20}}\put(120,0){\makebox(0,0){$K_{1}$}}
\put(130,0){\vector(1,0){20}}\put(160,0){\makebox(0,0){$K_{0}$}}
\put(170,0){\vector(1,0){20}}\put(200,0){\makebox(0,0){$\mathbbm{Z}$}}
\put(210,0){\vector(1,0){20}}\put(240,0){\makebox(0,0){$0$}}

\put(0,-40){\makebox(0,0){$ \cdots$}}\put(10,-40){\vector(1,0){20}}\put(40,-40){\makebox(0,0){$B_{3}$}}
\put(50,-40){\vector(1,0){20}}\put(80,-40){\makebox(0,0){$B_{2}$}}
\put(90,-40){\vector(1,0){20}}\put(120,-40){\makebox(0,0){$B_{1}$}}
\put(130,-40){\vector(1,0){20}}\put(160,-40){\makebox(0,0){$B_{0}$}}
\put(170,-40){\vector(1,0){20}}\put(200,-40){\makebox(0,0){$\mathbbm{Z}$}}
\put(210,-40){\vector(1,0){20}}\put(240,-40){\makebox(0,0){$0$}}

\put(40,-10){\vector(0,-1){20}}
\put(80,-10){\vector(0,-1){20}}
\put(120,-10){\vector(0,-1){20}}
\put(158,-10){\line(0,-1){20}}\put(160,-10){\line(0,-1){20}}
\put(198,-10){\line(0,-1){20}}\put(200,-10){\line(0,-1){20}}

\put(60,5){\makebox(0,0){$d$}}
\put(100,5){\makebox(0,0){$d$}}
\put(140,5){\makebox(0,0){$d$}}

\put(60,-35){\makebox(0,0){$\partial_{3}$}}
\put(100,-35){\makebox(0,0){$\partial_{2}$}}
\put(140,-35){\makebox(0,0){$\partial_{1}$}}

\put(50,-20){\makebox(0,0){$F_{3}$}}
\put(90,-20){\makebox(0,0){$F_{2}$}}
\put(130,-20){\makebox(0,0){$F_{1}$}}
\end{picture}
\end{figure}
\end{lemma}
\begin{proof} The proof is routine and indeed becomes much easier, so we omit it.
\end{proof}

\begin{corollary}\label{c1.7} Let $\phi\in B_{3}^{\ast}$ be a $3$-cocycle. Then $\phi$ is a
$3$-coboundary if and only if $F_{3}^{\ast}(\phi)$ is a $3$-coboundary.\end{corollary}
\begin{proof} Follows from the fact that $F_{3}^{\ast}$ induces an isomorphism between $3$-cohomology groups.
\end{proof}

\subsection{Abelian cocycles}\label{s3.2} We start with the definition of abelian cocycles.
For this, we need to recall the definition
of the twisted quantum double. The {\em twisted quantum double} $D^{\Phi}(G)$ of $G$ with respect to the $3$-cocycle
$\Phi$ over $G$ is the semisimple quasi-Hopf algebra with underlying vector
space $(\k G)^{\ast} \otimes \k G$  in which multiplication,
comultiplication $\Delta$, associator $\phi$, counit $\varepsilon$, antipode
$\S$, $\alpha$ and $\beta$ are given by
\begin{eqnarray*}
&&(e(g) \otimes x)(e(h) \otimes y) =\theta_g(x,y) \delta_{g^x,h}
e(g)\otimes x y,\\
&&\Delta(e(g)\otimes x)  = \sum_{hk=g} \gamma_x(h,k) e(h)\otimes x
\otimes e(k) \otimes x,\\
&&\phi = \sum_{g,h,k \in G} \Phi(g,h,k)^{-1} e(g) \otimes 1 \otimes
e(h) \otimes 1\otimes e(k) \otimes 1,\\
&&\S(e(g)\otimes x) =
\theta_{g^{-1}}(x,x^{-1})^{-1}\gamma_x(g,g^{-1})^{-1}e(x^{-1}g^{-1}x)\otimes
x^{-1},\\
&&\varepsilon(e(g)\otimes x) = \delta_{g,1}, \quad \alpha=1, \quad \beta=\sum_{g \in
G} \Phi(g,g^{-1},g)e(g)\otimes 1,
\end{eqnarray*}
where $\{e(g)|g\in G\}$ is the dual basis of $\{g|g\in G\}$,  $\delta_{g,1}$ is the Kronecker delta,  $g^x=x^{-1}g x$, and
\begin{eqnarray*}
\theta_g(x,y) &=&\frac{\Phi(g,x,y)\Phi(x,y,(x y)^{-1}g x y)}{\Phi(x,x^{-1}g x,y)}, \\
\gamma_g(x,y) & = & \frac{\Phi(x,y,g)\Phi(g, g^{-1}x g, g^{-1}yg)}{\Phi(x,g,
g^{-1}y g)}
\end{eqnarray*}
for any $x, y, g \in G$ (cf. \cite{DPR}).

Clearly,  $M$ is a
left $D^{\Phi}(G)$-module if and only if $M$ is a left-left Yetter-Drinfeld module over $(\k G,\Phi)$ as defined in the previous section. For our purpose, we prefer the following equivalent definition of abelian cocycles via twisted quantum doubles appeared in \cite{MN}.

\begin{definition} 
A $3$-cocycle  $\Phi$ over $G$ is called \emph{abelian} if $D^{\Phi}(G)$ is a commutative algebra.
\end{definition}

\begin{remark} 
Abelian cocycles of the previous form and  some related properties were discussed by Ng \cite{Ng} and by Ng-Mason \cite{MN}. In \cite{Ng}, Ng gave a quite symmetric description of abelian cocycles. Note that the Eilenberg-Mac Lane abelian cocycles \cite{EM} are different from the present ones. Recall that, an Eilenberg-Mac Lane abelian cocycle is a pair $(\Phi, d)$ where $\Phi\in \Z^{3}(G,\k^\ast)$ and $d$ is a braiding which is compatible with $\Phi$. But, we still have the following observation: if $(\Phi,d)$ is an Eilenberg-Mac Lane abelian cocycle, then $\Phi$ must be an abelian cocycle in our sense. As this fact is not necessary for our following discussions, here we won't provide a proof.
\end{remark}

As a direct consequence of this definition, we have the following conclusion.
\begin{corollary}\label{c3.5} Every Yetter-Drinfeld module over $(\k G,\Phi)$ is diagonal if and only if $\Phi$ is abelian.
\end{corollary}

Now we go back to the situation where $\G$ is an abelian group. So $\G\cong \mathbb{Z}_{\m_{1}}\times\cdots \times\mathbb{Z}_{\m_{n}}$ with $\m_j\in \mathbb{N}$
for $1\leq j\leq n$ and $\m_{i}|\m_{i+1}$
for all $1\leq i\leq n-1$. Let $\g_i$ be a generator of $\mathbb{Z}_{\m_{i}}$. By Proposition \ref{p3.8}, we can assume that $\Phi=\Phi_{\underline{\textbf{a}}}$ for some $\textbf{a} \in A.$
Using our formula of $3$-cocycles, we have the following conclusion which provides us a quite explicit description of abelian cocycles.

\begin{proposition}\label{p3.9} The $3$-cocycle $\Phi_{\underline{\mathbf{a}}}$ is  abelian if and only if $a_{rst}=0$
for all $1\leq r<s<t\leq n$.
\end{proposition}
\begin{proof} ``$\Leftarrow$:"$\;\;$ If all $a_{rst}=0$, then by \eqref{3cocycle} it is not hard to find that
$$\Phi_{\underline{\mathbf{a}}}{(x,y,z)}=\Phi_{\underline{\mathbf{a}}}(x,z,y)$$
for $x,y,z\in \G$. From this, we can find that $$\theta_g(x,y)=\theta_g(y,x)$$
for $g,x,y\in \G$, which implies that $D^{\Phi_{\underline{\mathbf{a}}}}(\G)$ is a commutative.

``$\Rightarrow:$"$\;\;$ If $a_{rst}\neq 0$ for some $r<s<t$. For simple, assume that $a_{123}\neq 0$. Through direct
computations, we have
$$\theta_{\g_1}(\g_2,\g_3)=1,\;\;\;\;\;\theta_{\g_1}(\g_3,\g_2)=\zeta_{\m_{1}}^{-a_{123}}.$$
This implies that $$(e(\g_1)\otimes \g_2)(e(\g_1)\otimes \g_3)\neq (e(\g_1)\otimes \g_3)(e(\g_1)\otimes \g_2).$$
\end{proof}

\subsection{Resolution} Let $\G=\Z_{\m_1}\times\cdots\times \Z_{\m_n}$ as before and $\Phi_{\underline{\mathbf{a}}}$ be an abelian $3$-cocycle of $\G$.
One of our key observations is that $\Phi_{\underline{\mathbf{a}}}$ can be ``resolved" in
 a slightly bigger abelian group $G$. More precisely, take $G=\mathbb{Z}_{m_1}\times \cdots\times\mathbb{Z}_{m_n}$
for $m_i=\mathbbm{m}_i^2\; (1\le i\le n)$. As before,  let ${\g}_{i}$ (resp. ${g}_{i}$) be a generator of $\mathbb{Z}_{\m_i}$
(resp. $\mathbb{Z}_{m_i}$) for $1\le i\le n$. Using such notations, we have a canonical group epimorphism:
$$\pi:\;G\to \mathbbm{G},\;\;\;\;g_{i}\mapsto \mathbbm{g}_{i}\; (1\le i\le n).$$
From this map, we can pull back the $3$-cocycles of $\mathbb{G}$ and get many $3$-cocycles over $G$. That
is,  the map
$$\pi^{\ast}(\Phi_{\underline{\mathbf{a}}}):\;G\times G\times G\to k^{\ast},\;\;(g,h,z)\mapsto \Phi_{\underline{\mathbf{a}}}(\pi(g),\pi(h),\pi(z)),\;\;\;g,h,z\in G$$
is a $3$-cocycle of $G$. Our observation is that $\pi^{\ast}(\Phi_{\underline{\mathbf{a}}})$ is indeed a boundary.
In fact, consider the following map
\begin{equation}\label{2cochain} J_{\underline{\mathbf{a}}}:\; G\times G\to k^{\ast};\;\;\;\;(g_{1}^{x_{1}}\cdots g_{n}^{x_{n}},g_{1}^{y_{1}}\cdots g_{n}^{y_{n}})\mapsto
\prod_{l=1}^{n}\zeta_{m_l}^{a_{l}x_{l}(y_l-y_l')}
\prod_{1\leq s<t\leq n}\zeta_{\m_s\m_{t}}^{a_{st}x_{t}(y_s-y_s')}
\end{equation}
where $y_i'$ is the remainder of $y_i$ divided by $\mathbbm{m}_i$ for $1\le i\le n$.
Here for simple, we just take $\zeta_{t}=e^{\frac{2\pi i}{t}}$ for $t\in \mathbb{N}$. Thus, we have

\begin{proposition}\label{p3.10} The differential of $J_{\underline{\mathbf{a}}}$ equals to $\pi^{\ast}(\Phi_{\underline{\mathbf{a}}})$, that is
$$\partial(J_{\underline{\mathbf{a}}})=\pi^{\ast}(\Phi_{\underline{\mathbf{a}}}).$$
\end{proposition}
\begin{proof} Indeed,

$\partial(J_{\underline{\mathbf{a}}})(g_{1}^{i_1}\cdots g_{n}^{i_n},g_{1}^{j_1}\cdots g_{n}^{j_n},g_{1}^{k_1}\cdots g_{n}^{k_n})$
\begin{eqnarray*}
&=&\frac{J_{\underline{\mathbf{a}}}(g_{1}^{j_1}\cdots g_{n}^{j_n},g_{1}^{k_1}\cdots g_{n}^{k_n})
J_{\underline{\mathbf{a}}}(g_{1}^{i_1}\cdots g_{n}^{i_n},g_{1}^{j_1+k_1}\cdots g_{n}^{j_n+k_n})}{J_{\underline{\mathbf{a}}}(g_{1}^{i_1+j_1}\cdots g_{n}^{i_n+j_n},g_{1}^{k_1}\cdots g_{n}^{k_n})J_{\underline{\mathbf{a}}}(g_{1}^{i_1}\cdots g_{n}^{i_n},g_{1}^{j_1}\cdots g_{n}^{j_n})}\\
&=&\frac{\prod_{l=1}^{n}\zeta_{m_l}^{a_{l}j_{l}(k_l-k_l')}
\prod_{1\leq s<t\leq n}\zeta_{\m_s\m_{t}}^{a_{st}j_{t}(k_s-k_s')}\prod_{l=1}^{n}\zeta_{m_l}^{a_{l}i_{l}(j_l+k_l-(j_l+k_l)')}
\prod_{1\leq s<t\leq n}\zeta_{\m_s\m_{t}}^{a_{st}i_{t}(j_s+k_s-(j_s+k_s)')}}
{\prod_{l=1}^{n}\zeta_{m_l}^{a_{l}(i_{l}+j_l)(k_l-k_l')}
\prod_{1\leq s<t\leq n}\zeta_{\m_s\m_{t}}^{a_{st}(i_{t}+j_t)(k_s-k_s')}\prod_{l=1}^{n}\zeta_{m_l}^{a_{l}i_{l}(j_l-j_l')}
\prod_{1\leq s<t\leq n}\zeta_{\m_s\m_{t}}^{a_{st}i_{t}(j_s-j_s')}}\\
&=&\prod_{l=1}^{n}\zeta_{m_l}^{a_{l}i_{l}(j_l'+k_l'-(j_l+k_l)')}
\prod_{1\leq s<t\leq n}\zeta_{\m_s\m_{t}}^{a_{st}i_{t}(j_s'+k_s'-(j_s+k_s)')}\\
&=&\prod_{l=1}^{n}\zeta_{\m_l}^{a_{l}i'_{l}[\frac{j_l'+k_l'}{\m_l}]}
\prod_{1\leq s<t\leq n}\zeta_{\m_{t}}^{a_{st}i'_{t}[\frac{j_s'+k_s'}{\m_s}]}
\\
&=&\pi^{\ast}(\Phi_{\underline{\mathbf{a}}})(g_{1}^{i_1}\cdots g_{n}^{i_n},g_{1}^{j_1}\cdots g_{n}^{j_n},g_{1}^{k_1}\cdots g_{n}^{k_n}).
\end{eqnarray*}
\end{proof}

Although the above conclusion is true for abelian $3$-cocycles, it does not hold for non-abelian $3$-cocycles. Precisely, let $\Phi$ be a non-abelian $3$-cocycle on $\G$. Then we will show that there does not exist any finite abelian group  $G'$ such that there is a group epimorphism $\pi:\; G'\to
\G$ making $\pi^{*}(\Phi)$ to be a coboundary (this is a surprising phenomenon, at least, to us). To prove this fact, we start with the following special case at first, then reduce the general case to this special case.

\begin{lemma}\label{sc} Let $\Phi_{\underline{\mathbf{a}}}$ be a non-abelian $3$-cocycle on $\G$. Suppose that we have a group epimorphism $$\pi\colon\Z_{l_1}\times\cdots \times \Z_{l_n}=\langle g_1\rangle \times\cdots\times \langle g_n\rangle \to \G=\langle \g_1\rangle \times\cdots\times \langle \g_n\rangle,\;\;\;\;g_i\mapsto \g_i,$$
then $\pi^{*}(\Phi_{\underline{\mathbf{a}}})$ is not a coboundary on $\Z_{l_1}\times\cdots \times \Z_{l_n}$.
\end{lemma}
\begin{proof} Since $\Phi_{\underline{\mathbf{a}}}$ is not an abelian cocycle, there are $r<s<t$ such that $a_{rst}\neq 0$ by Proposition \ref{p3.9}. Without loss of generality, we assume that $a_{123}\neq 0$. Assume that $\pi^{*}(\Phi_{\underline{\mathbf{a}}})$ is a coboundary.  By Corollary \ref{c1.7}, $F_{3}^{*}(\pi^{*}(\Phi_{\underline{\mathbf{a}}}))$ is coboundary and then Lemma \ref{l1.5} implies that $F_{3}^{*}(\pi^{*}(\Phi_{\underline{\mathbf{a}}}))_{1,2,3}=1$. But, direct computation shows that
\begin{eqnarray*}
F_{3}^{*}(\pi^{*}(\Phi_{\underline{\mathbf{a}}}))_{1,2,3}&=&F_{3}^{*}(\pi^{*}(\Phi_{\underline{\mathbf{a}}}))(\Psi_{1,2,3})\\
&=&\pi^{*}(\Phi_{\underline{\mathbf{a}}})([g_1,g_2,g_3]-[g_2,g_1,g_3]-[g_1,g_3,g_2]+[g_3,g_1,g_2]+[g_2,g_3,g_1]
-[g_3,g_2,g_1])\\
&=&\Phi_{\underline{\mathbf{a}}}([\g_1,\g_2,\g_3]-[\g_2,\g_1,\g_3]-[\g_1,\g_3,\g_2]+[\g_3,\g_1,\g_2]+[\g_2,\g_3,\g_1]
-[\g_3,\g_2,\g_1])\\
&=&\zeta_{\m_1}^{a_{123}}\neq 1.
\end{eqnarray*}
That's a direct contradiction.
\end{proof}
\begin{proposition} Let $\Phi_{\underline{\mathbf{a}}}$ be a non-abelian $3$-cocycle on $\G$ and $G$ be an arbitrary finite abelian group. Suppose that we have a group epimorphism $\pi\colon G \twoheadrightarrow \G$,
then $\pi^{*}(\Phi_{\underline{\mathbf{a}}})$ is not a coboundary on $G$.
\end{proposition}
\begin{proof} On the contrary, assume that $\pi^{*}(\Phi_{\underline{\mathbf{a}}})$ is a coboundary on $G$. Let $g_i$ be a preimage of $\g_i$ for $1\le i\le n$. Let $G_1$ be the subgroup generated by $g_1,\ldots, g_n$ and so we have a group embedding $\iota\colon G_1\to G$. Assume that $\ord(g_i)=l_i$, then clearly we have the following group epimorphism
$$\pi'\colon \Z_{l_1}\times\cdots \times \Z_{l_n}=\langle h_1\rangle\times \cdots\times \langle h_n\rangle \twoheadrightarrow G_1,\;\;\;\;\;h_i\mapsto g_i.$$
Define $f_1:=\iota\circ \pi'\colon \Z_{l_1}\times\cdots \times \Z_{l_n}\to G$ and $f:=\pi\circ f_1$. Note that by definition, the map $f$ is given by
$$\Z_{l_1}\times\cdots \times \Z_{l_n}\to \G,\;\;\;\;\;g_i\mapsto \g_i,\;\;(1\le i\le n).$$
If $\pi^{*}(\Phi_{\underline{\mathbf{a}}})$ is a coboundary on $G$, then $f_1^{*}(\pi^{*}(\Phi_{\underline{\mathbf{a}}}))=f^{*}(\Phi_{\underline{\mathbf{a}}})$ is still a coboundary. But this is absurd by Lemma \ref{sc}.

\end{proof}
\section{Nichols algebras of diagonal type in $_\mathbbm{G}^\mathbbm{G}\mathcal{YD}^\Phi.$}

 The aim of this section is to give a classification of the Nichols algebras of diagonal type with arithmetic root system in $_\mathbbm{G}^\mathbbm{G}\mathcal{YD}^\Phi.$ The idea to realize our purpose consists of five steps. Firstly, we can assume that the support group of $\B(V)$ is $\mathbbm{G},$ and from this assumption we can prove that $\Phi$ must be an abelian $3$-cocycle over $\G$. Secondly, we will develop a technique to change the base group from $\G$ to a bigger one $G$ together with a group epimorphism $\pi: G\to \G.$ Thirdly, we will show that any Nichols algebra $\B(V)$ in $_\mathbbm{G}^\mathbbm{G}\mathcal{YD}^\Phi$ is isomorphic to a Nichols algebra in $_G^G\mathcal{YD}^{\pi^*(\Phi)},$ which is thus twist equivalent to a usual Nichols algebra by Proposition \ref{p3.10}. Fourthly, we want to get a return ticket, that is, we will give a sufficient and necessary condition to determine when a Nichols algebra in $_\mathbbm{G}^\mathbbm{G}\mathcal{YD}^{\pi^*(\Phi)}$ is isomorphic to one in $_\mathbbm{G}^\mathbbm{G}\mathcal{YD}^\Phi$. Finally, combining these results and Heckenberger's classification of arithmetic root systems, we obtain the classification of Nichols algebras of diagonal type with arithmetic root system in $_\mathbbm{G}^\mathbbm{G}\mathcal{YD}^\Phi.$

\subsection{Start points} We give two conclusions as our preparations for classification.
At first, we will prove that any Nichols algebras of diagonal type can be realized in $_\mathbbm{G}^\mathbbm{G} \mathcal{YD}^\Phi$, where $\mathbbm{G}$ is an abelian group and $\Phi$ is an abelian $3$-cocycle over $\mathbbm{G}.$

\begin{lemma}\label{P4.7}
Let $\B(V)$ be a Nichols algebra of diagonal type in $_\mathbbm{G}^\mathbbm{G} \mathcal{YD}^\Phi,$ and $\mathbbm{G}'$ be the support group of $\B(V).$ Let $\Psi=\Phi|_{\mathbbm{G}'}$, then $\mathbbm{G}'$ is an abelian group and $\Psi$ is an abelian $3$-cocycle over $\mathbbm{G}'$.
\end{lemma}
\begin{proof}
Firstly we will prove that $\mathbbm{G}'$ is an abelian group.

Suppose $V=\k \{X_1,\cdots ,X_m\}$ and $\delta_l(X_i)=h_i\otimes X_i, \;\forall 1\leq i\leq m.$ Since $\mathbbm{G}'$ is a support group of $\B(V)$,  $\mathbbm{G}'=\langle h_1,\cdots,h_m\rangle$. In order to prove that $\mathbbm{G}'$ is abelian, we need to prove $h_ih_j=h_jh_i, \forall 1\leq i\neq j\leq m$.
Let $U=\k\{X_i,X_j\}, H=\langle h_i,h_j\rangle,$ then $\B(U)$ is a Nichols algebra in $_{\k H}^{\k H} \mathcal{YD}^{\Psi|_H}.$ Therefore, $\B(U)\# \k H$ is a rank $2$ connected graded pointed Majid algebras. According to \cite[Lemma 2.7]{qha6}, $H$ is an abelian group, so $h_ih_j=h_jh_i.$ We have proved that $\mathbbm{G}'$ is an abelian group.

Next we will prove that $\Psi$ is an abelian $3$-cocycle over $\mathbbm{G}'.$ Assume without loss of generality that $\mathbbm{G}'=\mathbb{Z}_{m_1}\times\cdots\times \mathbb{Z}_{m_n} = \langle g_1 \rangle \times \cdots \langle g_n\rangle$, and $\Psi$ is of the form \eqref{3cocycle}.
According to Proposition \ref{p3.9}, we only need to prove that $a_{rst}=0,\forall 1\leq r<s<t\leq n.$

At first, fix a triple $(r,s,t)$ such that $1\leq r<s<t\leq n.$ Since $\G'=\langle h_1,\cdots,h_m\rangle$,  we have
$g_r=h_1^{k_{1}}\cdots h_m^{k_{m}}$, where $k_{1}<|h_1|,\cdots, k_{m}<|h_m|.$ Here $|g|$ means the order of $g.$
Conversely, $h_i,1\leq i\leq m$ can be presented by the generators of $\G'$, i.e. $h_i=g_1^{c_{i1}}\cdots \ g_n^{c_{in}}$, and we get
\begin{equation*}
\sum_{i=1}^m k_{i}c_{il} \equiv \begin{cases} 0 \; (\textrm{mod} \;m_l), & l\neq r; \\ 1 \; (\textrm{mod}\; m_r), & l=r.
\end{cases}
\end{equation*}
By \eqref{3cocycle}£¬
\begin{equation*}
\Psi(h_i,g_s,g_t)
=\prod_{1\leq j<s}\zeta_{(m_j,m_s,m_t)}^{c_{ij}a_{jst}},
\end{equation*}
so $$\prod_{i=1}^{m}\Psi(h_i,g_s,g_t)^{k_i}= \prod_{1\leq j<s}\zeta_{(m_j,m_s,m_t)}^{a_{jst}(\sum_{i=1}^nk_ic_{ij})}=\zeta_{(m_r,m_s,m_t)}^{a_{rst}}.$$

On the other hand, since $\k X_i,1\leq i\leq m$ are one dimensional $(\k \mathbbm{G}',\widetilde{\Psi}_{h_i})$-representations,  $\widetilde{\Psi}_{h_i}(g_s,g_t)=\widetilde{\Psi}_{h_i}(g_t,g_s)$. It follows by a direct computation that
$\widetilde{\Psi}_{h_i}(g_s,g_t)=1$ (since $s<t$), and thus $$\widetilde{\Psi}_{h_i}(g_t,g_s)=\widetilde{\Psi}_{h_i}(g_s,g_t)=1.$$

Hence
\begin{equation*}
\Psi(h_i,g_s,g_t)=\frac{\Psi(g_s,h_i,g_t)}{\Psi(g_s,g_t,h_i)}
=\prod_{s<p<t}\zeta_{(m_s,m_p,m_t)}^{c_{ip}a_{spt}}\bigg[\prod_{t<q\leq n}\zeta_{(m_s,m_t,m_q)}^{c_{iq}a_{stq}}\bigg]^{-1}.
\end{equation*}
Then we get
\begin{equation*}
\begin{split}
\prod_{i=1}^{m}\Psi(h_i,g_s,g_t)^{k_i}=&\prod_{i=1}^m\bigg
\{\prod_{s<p<t}\zeta_{(m_s,m_p,m_t)}^{c_{ip}a_{spt}}
\bigg[\prod_{t<q\leq n}
\zeta_{(m_s,m_t,m_q)}^{c_{iq}a_{stq}}\bigg]^{-1}\bigg\}^{k_i}\\
=&\prod_{s<p<t}\zeta_{(m_s,m_p,m_t)}^{a_{spt}(\sum_{i=1}^m k_ic_{ip})}\bigg[\prod_{t<q\leq n}\zeta_{(m_s,m_t,m_q)}^{a_{stq}(\sum_{i=1}^m k_ic_{iq})}\bigg]^{-1}\\
=&1.
\end{split}
\end{equation*}
So we obtain $\zeta_{(m_r,m_s,m_t)}^{a_{rst}}=1$ and this implies $a_{rst}=0$ since $0\leq a_{rst}<(m_r,m_s,m_t)$.
\end{proof}

For our purpose of classification of Nichols algebras $\B(V)$ of diagonal type in ${_G^G \mathcal{YD}^\Phi}$, it is harmless to \emph{assume in the rest of the paper that the support group of $\B(V)$ is $G$ and thus $G$ is abelian and $\Phi=\Phi_{\underline{\textbf{a}}}$ is an abelian cocycle}.

Secondly, we will show that there is a nice grading on $\B(V)\in {_G^G \mathcal{YD}^\Phi}$. Let $V\in {_G^G \mathcal{YD}^\Phi}$ be a Yetter-Drinfeld module of diagonal type and $\{X_i|1\leq i\leq l\}$ a canonical basis of $V.$ Let $\Z^l$ be the free abelian group of rank $l$ and assume that $e_i\; (1\leq i\leq l)$ is the canonical generators of $\Z^l$. The following fact is very important for our follow-up discussions, which is indeed \cite[Lemma 4.2]{qha6}. We include a proof here for completeness and safety.

\begin{proposition}\label{p4.1}
There is a $\Z^{l}$-grading on the Nichols algebra $\B(V)\in {_G^G \mathcal{YD}^\Phi}$ by setting $\deg X_i=e_i$.
\end{proposition}

\begin{proof}
Obviously, there is a $\Z^{l}$-grading on the tensor algebra $T_{\Phi}(V)\in {_G^G \mathcal{YD}^\Phi}$ by assigning $\deg X_i=e_i$. Let $I=\oplus_{i\geq i_0} I_i$ be the maximal graded Hopf ideal generated by homogeneous elements of degree greater than or equal to $2.$ To prove that $\B(V)$ is $\Z^l$-graded, it amounts to prove that $I$ is $\Z^l$-graded. This will be done by induction on the $\N$-degree.

First let $X\in I$ be a homogenous element with minimal degree $i_0.$ Since $\Delta(X)\in T_{\Phi}(V)\otimes I+ I\otimes T_{\Phi}(V),$ $X$ must be a primitive element, i.e., $\Delta(X)=X\otimes 1+1\otimes X.$
Suppose $X=X^1+X^2+\cdots +X^n,$ where $X^i$ is $\Z^l$-homogenous, and $X^i$ and $X^j$ have different $\Z^l$-degrees if $i\neq j.$ Suppose $\Delta(X^i)=X^i\otimes 1+1\otimes X^i+(X^i)_1\otimes (X^i)_2.$ Then we have  $\sum_{i=1}^n(X^i)_1\otimes(X^i)_2=0$. This forces $(X^i)_1\otimes (X^i)_2 = 0$ for each $1\leq i\leq l$ since $\Delta$ preserves the $\Z^l$-degrees. So, each $X^i$ is a primitive element and hence must be contained in $I$ by the maximality of $I.$ Therefore, $I_{i_0}$ is $\Z^l$-graded.

Then suppose that $I^k:=\oplus_{i_0\leq i\leq k} I_i$ is $\Z^l$-graded. We shall prove that $I^{k+1}=\oplus_{i_0\leq i\leq k+1} I_i$ is also $\Z^l$-graded. Let $X\in I_{k+1}$ and $X=X^1+X^2+\cdots +X^n,$ with each $X^i$ being $\Z^l$-homogenous and $X^i$ and $X^j$ having different $\Z^l$-degrees if $i\neq j.$ Write $\Delta(X^i)=X^i\otimes 1+1\otimes X^i+(X^i)_1\otimes (X^i)_2.$ Since $\Delta(X)=X\otimes 1+1\otimes X+(X)_1\otimes (X)_2,$ where $(X)_1\otimes (X)_2\in T_{\Phi}(V)\otimes I^k+I^k\otimes T_{\Phi}(V),$ i.e., $\sum (X^i)_1\otimes (X^i)_2\in T_{\Phi}(V)\otimes I^k+I^k\otimes T_{\Phi}(V).$ According to the inductive assumption, $T_{\Phi}(V)\otimes I^k+I^k\otimes T_{\Phi}(V)$ is a $\Z^l$-graded space. So each $(X^i)_1\otimes (X^i)_2\in T_{\Phi}(V)\otimes I^k+I^k\otimes T_{\Phi}(V)$ as $\Delta$ preserves $\Z^l$-degrees. If there was an $X^i\notin I_{k+1},$ then $I+\langle X^i\rangle$ is a Hopf ideal properly containing $I,$ which contradicts to the maximality of $I.$ It follows that $X^i\in I_{k+1}$ for all $1\leq i\leq n$ and hence $I^{k+1}$ is also $\Z^l$-graded by the assumption on $X.$ We complete the proof of the proposition.
\end{proof}

\subsection{Change of base groups}

Since Nichols algebras in twisted Yetter-Drinfeld categories $_\G^\G\mathcal{YD}^\Phi$ are nonassociative algebras, the structures of these algebras depend on $\G$ and $3$-cocycle $\Phi$ on $\G.$ We will call $\G$ the base group of $\B(V).$  One of the most important methods of this paper is to change the base groups of Nichols algebras. We need the following definition.

\begin{definition}
Let $\B(V)$ and $\B(U)$ be Nichols algebras in $_G^G\mathcal{YD}^\Phi$  and $_H^H\mathcal{YD}^\Psi$ respectively with $\dim V=\dim U=l.$ We say $\B(V)$ is isomorphic to $\B(U)$ if there is a $\Z^l$-graded linear isomorphism $\mathcal{F}:\B(V)\to \B(U)$ which preserves the multiplication and comultiplication.
\end{definition}

\begin{lemma}\label{p4.2}
Suppose $V\in {_G^G\mathcal{YD}^\Phi}$ and $U\in { _H^H\mathcal{YD}^\Psi}.$  Let $G'$ and $H'$ be support groups of $V$ and $U$ respectively. If there are a linear isomorphism $F:V\to U$ and a group epimorphism $f:G'\to H'$ such that:
\begin{eqnarray}
&\delta\circ F=(f\otimes F)\circ \delta,\label{eq4.1} \\
& F(g\triangleright v)=f(g)\triangleright F(v),\label{eq4.2}\\
& \Phi|_{G'}=f^*\Psi|_{H'}
\end{eqnarray}
for any $g\in G',$ $v\in V.$ Then $\B(V)$ is isomorphic to $\B(U).$
\end{lemma}
 \begin{proof}
Let $\mathcal{F}: T_{\Phi|_{G'}}(V)\to T_{f^*(\Psi|_{H'})}(U)$ be the multiplicative linear map such that $\mathcal{F}|_V=F.$ It is easy to show that $F$ also preserves the comultiplication between $T_{\Phi|_{G'}}(V)$ and $T_{f^*(\Psi|_{H'})}(U).$ Let $\{X_i|1\leq i\leq l\}$ be a canonical basis of $V,$ then it is obviously that $\{Y_i=F(X_i)|1\leq i\leq l\}$ is a canonical basis of $U$ by (\eqref{eq4.1}-\eqref{eq4.2}). Let $\{e_i|1\leq i\leq l\}$ be the free generators of $\Z^l.$ Then $T_{\Phi|_{G'}}(V)$ and $ T_{f^*(\Psi|_{H'})}(U)$ are $\Z^l$-graded by setting $\deg(X_i)=\deg(Y_i)=e_i.$
Note that $\mathcal{F}$ induces a one to one correspondence between the set of $\Z^l$-graded Hopf ideals of $T_{\Phi|_{G'}}(V)$ and that of $T_{f^*(\Psi|_{H'})}(U).$ By Proposition \ref{p4.1}, we know that the maximal Hopf ideals generated by homogeneous elements of degree $\geq 2$ in $ T_{\Phi|_{G'}}(V)$ and in $T_{f^*(\Psi|_{H'})}(U)$ are $\Z^l$-graded. It is obvious that $\mathcal{F}$ maps the maximal Hopf ideal of $T_{\Phi|_{G'}}(V)$ to that of $T_{f^*(\Psi|_{H'})}(U).$ Therefore, $\mathcal{F}$ induces a linear isomorphism from $\B(V)$ to $\B(U)$ which preserves multiplication and comultiplication.
\end{proof}

The following definition is convenient for our later expositions.
\begin{definition}
If $(F,f)$ is an isomorphism from $\B(V)$ to $\B(U)$ as in Lemma \ref{p4.2},  then we say $\B(V)$ is isomorphic to $\B(U)$ \emph{through the group morphism $f.$}\end{definition}

Suppose $\mathbb{G}=\mathbb{Z}_{\mathbbm{m}_1}\times\cdots\times \mathbb{Z}_{\mathbbm{m}_n} = \langle \mathbbm{g}_1 \rangle \times \cdots \langle \mathbbm{g}_n \rangle $ and $G=\mathbb{Z}_{m_1}\times\cdots\times \mathbb{Z}_{m_n} = \langle g_1 \rangle \times \cdots \langle g_n\rangle $ where $m_i=\mathbbm{m}^2_i$ for $1\leq i\leq n.$ Let
\begin{equation}\label{e4.4}\pi:\;\k G\to \k\mathbbm{G},\;\;\;\;g_{i}\mapsto \mathbbm{g}_{i},\;\;\;\;\;1\le i\le n\end{equation}
be the canonical epimorphism. Observe that $\pi$ has a section
\begin{equation}\label{e4.5}
\iota:\;\k\mathbbm{G}\to \k G,\;\;\;\;\prod_{i=1}^{n}\mathbbm{g}^{i_j}_{i}\mapsto \prod_{i=1}^{n}{g}^{i_j}_{i}
\end{equation}
which is not a group morphism.
Let $\delta_L$ and $\triangleright$ be the comodule and module structure maps of $V\in {_\mathbbm{G}^\mathbbm{G} \mathcal{YD}^{\Phi}}$. Define
\begin{eqnarray*}
&&\rho_{L}:\;V \to \k G\otimes V,\;\;\;\;\rho_{L}=(\iota\otimes \id)\delta_{L}\\
&&\blacktriangleright:\;\k G\otimes V\to V,\;\;\;\;g\blacktriangleright Z=\pi(g)\triangleright Z
\end{eqnarray*}
for all $g\in G$ and $Z\in V$.
 \begin{lemma}\label{l4.4} Defined in this way, $(V, \rho_{L}, \blacktriangleright),$ denoted simply by $\widetilde{V}$ in the following, is an object in ${_G^G \mathcal{YD}^{\pi^*\Phi}}$.
 \end{lemma} \begin{proof}
 This can be verified by direct computation:
\begin{eqnarray*}
e\blacktriangleright(f\blacktriangleright Z)&=&\pi(e)\triangleright(\pi(f)\triangleright Z)\\
&=&\tilde{\Phi}_{z}(\pi(e),\pi(f))(\pi(e)\pi(f))\triangleright Z\\
&=&\widetilde{\pi^{\ast}(\Phi)}_{\iota(z)}(e,f)ef\blacktriangleright Z
\end{eqnarray*}
for all $e,f\in G$ and $\delta_{L}(Z)=z\otimes Z$ for $Z\in V$.
\end{proof}
\begin{proposition}\label{P4.5}
For any Nichols algebra $\B(V)\in{ _\mathbbm{G}^\mathbbm{G} \mathcal{YD}^\Phi},$ the Nichols algebra $\B(\widetilde{V})\in {_G^G \mathcal{YD}^{\pi^*(\Phi)}}$ is isomorphic to $\B(V).$ Moreover, $\B(\widetilde{V})$ is twist equivalent to a usual Nichols algebra.
\end{proposition}
\begin{proof}
The first statement is a direct consequence of Lemma \ref{p4.2}. For the second, just note that $\pi^*(\Phi)$ is a 3-coboundary on $G$ by Proposition \ref{p3.10}.
\end{proof}

To summarize so far, we have found the following route of transforming a \emph{nonassociative} Nichols algebra to a \emph{usual} one:
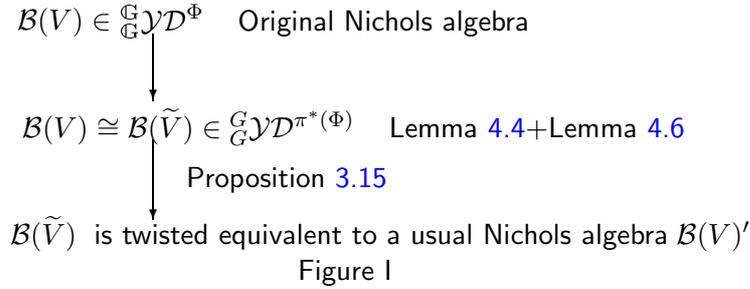
\begin{figure}[hbt]
\begin{picture}(100,100)(0,0)
\put(-5,90){\makebox(0,0){$\B(V)\in {^{\G}_{\G}\mathcal{YD}^\Phi}\;\;\;$ \textsf{Original Nichols algebra}}}
\put(-50,85){\vector(0,-1){25}}
\put(25,50){\makebox(0,0){$ \B(V)\cong \B(\widetilde{V})\in {^{G}_{G}\mathcal{YD}^{\pi^{\ast}(\Phi)}}\;\;\;$ \textsf{Lemma \ref{p4.2}+Lemma \ref{l4.4}}}}
\put(-50,45){\vector(0,-1){30}}\put(0,30){\makebox(0,0){\textsf{Proposition \ref{p3.10}}}}
\put(35,10){\makebox(0,0){$ \B(\widetilde{V})\;$ \textsf{is twisted equivalent to a usual Nichols  algebra}\;$\B(V)'$}}
\put(20,-5){\makebox(0,0){ \textsf{Figure I}}}
\end{picture}
\end{figure}

Since we only want to classify finite-dimensional Majid algebras, there is no harm to assume that all the usual Nichols algebras appearing in this paper have arithmetic root systems.  According to this diagram, each diagonal Nichols algebra $\B(V)\in {^{\G}_{\G}\mathcal{YD}^\Phi}$ is corresponding to a usual diagonal Nichols algebra, denoted by $\B(V)'$ for convenience, in a canonical way. Note that there is a $\Z^{l}$-graded linear isomorphism $\B(V)\cong \B(V)'$. Thus, it is reasonable to make the following definition.

\begin{definition}\label{d4.8} The arithmetic root system of $\B(V)$ is defined to be that of $\B(V)'.$ That is, $\triangle(\B(V))_{\chi, E} := \triangle(\B(V'))_{\chi, E}$ by the prescribed notations in Subsection \ref{subsec2.5}. In particular, the root system $\triangle(\B(V))$ of $\B(V)$ equals to $\triangle(\B(V)')$.
\end{definition}

The aim of this section is to classify the Nichols algebras of diagonal type with arithmetic root system in $_\mathbbm{G}^\mathbbm{G}\mathcal{YD}^\Phi.$ Using this diagram, we just need to know the following: For a usual diagonal type Nichols algebra $\B$ with arithmetic root system, when $\B$ is gotten from a Nichols algebra $\B(V)\in {_\mathbbm{G}^\mathbbm{G}\mathcal{YD}^\Phi}$? That is, find a return trip of the above diagram.

\subsection{The return trip} Keep the notations of the previous subsection.
At first, we give the inverse version of Proposition \ref{P4.5}.
\begin{lemma}\label{P4.6}
Let $\B(\widetilde{V})\in {_G^G \mathcal{YD}^{\pi^*(\Phi)}}$ be a Nichols algebra of diagonal type and $\{Y_i|1\leq i\leq m\}$ be a canonical basis of $\widetilde{V}.$ Then $\B(\widetilde{V})$ is isomorphic to a Nichols algebra in $_\mathbbm{G}^\mathbbm{G} \mathcal{YD}^\Phi$ through $\pi$ if and only if \begin{equation}g_i^{\m_i}\blacktriangleright Y_j=Y_j,\;\;\;\; 1\leq i\leq n, 1\leq j\leq m.\end{equation}
\end{lemma}
\begin{proof}
If $g_i^{\m_i}\blacktriangleright Y_j=Y_j,$ then one can easily show $\widetilde{V}$ is an object in $_\mathbbm{G}^\mathbbm{G} \mathcal{YD}^\Phi$ by defining:
\begin{eqnarray*}
&&\delta_{L}:\;V \to \k\G\otimes V,\;\;\;\;\delta_{L}=(\pi\otimes \id)\rho_{L};\\
&&\triangleright:\;\k\mathbbm{G}\otimes V\to V,\;\;\;\;\mathbbm{g}\triangleright Z=\iota(\mathbbm{g})\blacktriangleright Z.
\end{eqnarray*} It is obviously that $(\id_V,\pi)$ is an isomorphism between $\B(\widetilde{V})\in{_G^G \mathcal{YD}^{\pi^*(\Phi)}}$ and $\B(V)\in \ _\mathbbm{G}^\mathbbm{G} \mathcal{YD}^\Phi.$

Conversely, suppose $(F,\pi)$ is a isomorphism from $\B(V)\in{_G^G \mathcal{YD}^{\pi^*(\Phi)}}$ to $\B(U)\in \ _\mathbbm{G}^\mathbbm{G} \mathcal{YD}^\Phi.$ The by Definition \ref{d4.8} and Equation \eqref{eq4.2}, we have $ F(g_i^{\m_i}\blacktriangleright Y_j)=\pi(g_i^{\m_i})\blacktriangleright F(Y_j)=F(Y_j).$ This implies $g_i^{\m_i}\blacktriangleright Y_j=Y_j$ for all $1\leq i\leq n, 1\leq j\leq m.$
\end{proof}

Now fix a usual Nichols algebra of diagonal type $\B(V)'\in {_G^G\mathcal{YD}}$ with support group $G$. According to Figure I, we need to answer the following question:
\begin{center} When $\B(V)'^{J_{\underline{\textbf{a}}}}$ is isomorphic to a Nichols algebra in $_\mathbbm{G}^\mathbbm{G}\mathcal{YD}^{\Phi_{\underline{\mathbf{a}}}}$ through $\pi$?$\;\;\;\;\;\;$(\textsf{Question} $(\diamond)$)
\end{center}

Let $\{X_i|1\leq i\leq m\}$ be a canonical basis of $V.$  Assume that
$$\delta'_{L}(X_{i})=h_{i}\otimes X_{i},\;\;\;\;g_k\triangleright'X_j=q_{kj}X_j$$
for $1\leq i,j\leq m,\; 1\leq k\leq n$, $h_{i}\in G$ and $q_{kj}\in \k^*$, where $\delta'_L$ (resp. $\triangleright'$) is the comodule (resp. module)
structure map of $\B(V)'\in {^G_G\mathcal{YD}}$. So there are $0\leq x_{kj}, s_{ik}< m_{k}$ such that
$$q_{kj}=\zeta_{m_k}^{x_{kj}},\;\;h_{i}=\prod_{k=1}^{n}g_k^{s_{ik}}$$
for $1\leq i,j\leq m$ and $1\leq k\leq n$. Let $\mathbf{X}=(x_{ij})_{n\times m}$.
By assumption, the support group $G_{\B(V)'}=G$,  $\{h_{i}|1\leq i\leq m\}$ generate the group $G,$  so there are $t_{jl}\in \mathbb{N}$ such that
$$g_{j}=\prod_{l=1}^{m}h_l^{t_{jl}},\quad 1\leq j\leq n.$$
By $\mathbf{S}$ and $\mathbf{T},$ we denote the matrixes $(s_{ik})_{m\times n}$ and $(t_{jl})_{n\times m}.$ It is obvious that
\begin{equation} \label{e4.6} \mathbf{T}\mathbf{S}
  \equiv \left(\begin{array}{cccc}1\ (\text{mod} \ m_1) & 0 \ (\text{mod}\ m_1)&\cdots& 0 \ (\text{mod}\ m_1)\\ 0\ (\text{mod}\ m_2)& 1\ (\text{mod}\ m_2)&\cdots&  0\ (\text{mod}\ m_2)\\
  \cdots&\cdots&\cdots&\cdots\\
  0\ (\text{mod}\ m_n)& 0\ (\text{mod}\ m_n)&\cdots&  1\ (\text{mod}\ m_n)
  \end{array}\right).\end{equation}
 With these notations, we can give the answer to Question $(\diamond)$ now.

\begin{proposition}\label{P4.8}
The twisting $\B(V)'^{J_{\underline{\textbf{a}}}}$ is isomorphic to a Nichols algebra in $_\mathbbm{G}^\mathbbm{G}\mathcal{YD}^{\Phi_{\underline{\mathbf{a}}}}$ through $\pi$ if and only if the following congruence equalities hold:
\begin{eqnarray}
\sum_{j=1}^m x_{ij}t_{lj}&\equiv& 0 \;(\emph{mod} \;\mathbbm{m}_i), \;1\leq l<i \leq n,\label{e4.7}\\
\sum_{j=1}^mx_{ij}t_{ij}&\equiv& a_i \;(\emph{mod} \;\mathbbm{m}_i), \;1\leq i\leq n,\label{e4.8}\\
(\sum_{j=1}^mx_{ij}t_{lj})\mathbbm{m}_l&\equiv& \mathbbm{m}_ia_{il}\; (\emph{mod} \; \mathbbm{m}_i\mathbbm{m}_l), \;1\leq i<l\leq n.\label{e4.9}
\end{eqnarray}
\end{proposition}

\begin{proof}
By Lemma \ref{P4.6}, $\B(V)'^{J_{\underline{\mathbf{a}}}}$ is isomorphic to a Nichols algebra in $_\mathbbm{G}^\mathbbm{G}\mathcal{YD}^{\Phi_{\underline{\mathbf{a}}}}$ if and only if $g_i^{\m_i}\triangleright'_{J_{\underline{\mathbf{a}}}} X_j=X_j$ for all $1\leq i\leq n, 1\leq j\leq m.$ By definition, we have
$$
g_i^{\mathbbm{m}_i}\triangleright'_{J_{\underline{\mathbf{a}}}} X_j=\frac{J_{\underline{\mathbf{a}}}(g^{\mathbbm{m}_i}_i,h_j)}
{J_{\underline{\mathbf{a}}}(h_j,g^{\mathbbm{m}_i}_i)}g^{\mathbbm{m}_i}_i\triangleright' X_j=\frac{J_{\underline{\mathbf{a}}}(g^{\mathbbm{m}_i}_i,h_j)}
{J_{\underline{\mathbf{a}}}(h_j,g^{\mathbbm{m}_i}_i)}\zeta_{m_i}^{\mathbbm{m}_i
x_{ij}}X_j
$$ for all $1\leq i\leq n, 1\leq j\leq m.$ Using \eqref{2cochain},
\begin{equation*}
\frac{J_{\underline{\mathbf{a}}}(g^{\mathbbm{m}_i}_i,h_j)}{J_{\underline{\mathbf{a}}}(h_j,g^{\mathbbm{m}_i}_i)}\zeta_{m_i}^{\mathbbm{m}_ix_{ij}}=
\frac{1}{\zeta_{m_i}^{a_is_{ji}\mathbbm{m}_i}\prod_{i<k\leq n}\zeta_{\mathbbm{m}_i\mathbbm{m}_k}^{a_{ik}s_{ik}\mathbbm{m}_i}}\zeta_{m_i}^{\mathbbm{m}_i
x_{ij}}.
\end{equation*}
So for all $1\leq i\leq n, 1\leq j\leq m,$  equations $g_i^{\m_i}\triangleright'_{J_{\underline{\mathbf{a}}}} X_j=X_j$ are equivalent to
\begin{equation}\label{e4.10}
\zeta_{m_i}^{\mathbbm{m}_i
x_{ij}}=\zeta_{m_i}^{a_is_{ji}\mathbbm{m}_i}\prod_{i<k\leq n}\zeta_{\mathbbm{m}_i\mathbbm{m}_k}^{a_{ik}s_{jk}\mathbbm{m}_i}.
\end{equation}
Next we will show that Equations \eqref{e4.10} are equal to Equations \eqref{e4.7}-\eqref{e4.9}. At first, we assume that \eqref{e4.10} hold. Then for any $1\leq l\leq n,$ we have
\begin{equation}\label{e4.11}
\zeta_{m_i}^{\mathbbm{m}_i
x_{ij}t_{lj}}=\zeta_{m_i}^{a_it_{lj}s_{ji}\mathbbm{m}_i}\prod_{i<k\leq n}\zeta_{\mathbbm{m}_i\mathbbm{m}_k}^{a_{ik}t_{lj}s_{jk}\mathbbm{m}_i}.
\end{equation}
Consider the product of two sides of equations \eqref{e4.11} for $j=1,\cdots ,m,$ we get
 \begin{equation}\label{e4.12}
\prod_{j=1}^{m}\zeta_{m_i}^{\mathbbm{m}_i
x_{ij}t_{lj}}=\prod_{j=1}^{m}[\zeta_{m_i}^{a_it_{lj}s_{ji}\mathbbm{m}_i}\prod_{i<k\leq n}\zeta_{\mathbbm{m}_i\mathbbm{m}_k}^{a_{ik}t_{lj}s_{jk}\mathbbm{m}_i}],
\end{equation} which is equal to
\begin{equation}\label{e4.13}
\zeta_{m_i}^{\sum_{j=1}^{m}\mathbbm{m}_i
x_{ij}t_{lj}}=\zeta_{m_i}^{a_i\sum_{j=1}^{m}t_{lj}s_{ji}\mathbbm{m}_i}\prod_{i<k\leq n}\zeta_{\mathbbm{m}_i\mathbbm{m}_k}^{a_{ik}\sum_{j=1}^{m}t_{lj}s_{jk}\mathbbm{m}_i}.
\end{equation}
When $i>l$, by \eqref{e4.6} the Equations \eqref{e4.13} become \begin{equation}\label{e4.14}\zeta_{m_i}^{\sum_{j=1}^{m}\mathbbm{m}_i
x_{ij}t_{lj}}=1.\end{equation} These imply that $\sum_{j=1}^m x_{ij}t_{lj}\equiv 0\; (\textrm{mod}\;\mathbbm{m}_i)$, which are the Equations \eqref{e4.7}.
When $i=l,$ then \eqref{e4.13} become \begin{equation}\label{e4.15}\zeta_{m_i}^{\sum_{j=1}^{m}\mathbbm{m}_i
x_{ij}t_{ij}}=\zeta_{m_i}^{a_i\mathbbm{m}_i},\end{equation} which imply the Equations \eqref{e4.8}. When $i<l,$ then \eqref{e4.13} become \begin{equation}\label{e4.16}\zeta_{m_i}^{\sum_{j=1}^{m}\mathbbm{m}_i
x_{ij}t_{lj}}=\zeta_{\mathbbm{m}_i\mathbbm{m}_l}^{a_{il}\mathbbm{m}_i},\end{equation} which are the the same as the Equations \eqref{e4.9}.

Next, we assume that \eqref{e4.7}-\eqref{e4.9} hold. Clearly, the Equations \eqref{e4.7}-\eqref{e4.9} are equal to \eqref{e4.14}-\eqref{e4.16}. Consider the product of the $3$-identities, we get
 \begin{equation}\label{eq4.18}
\prod_{j=1}^{m}\zeta_{m_i}^{\mathbbm{m}_i
x_{ij}t_{lj}}=\prod_{j=1}^{m}[\zeta_{m_i}^{a_it_{lj}s_{ji}\mathbbm{m}_i}\prod_{i<k\leq n}\zeta_{\mathbbm{m}_i\mathbbm{m}_k}^{a_{ik}t_{lj}s_{jk}\mathbbm{m}_i}]
\end{equation} for $1\leq i,l\leq n.$
Let $\mathbf{T}'$ be an $m\times m$-matrix such that
$t'_{ij}=t_{ij}$ for $ 1\leq i\leq n, 1\leq j\leq m$ and otherwise $t'_{ij}=0.$
From \eqref{e4.6} we know that the rank of $\mathbf{T}'$ is $n.$ So for any $1\leq i\leq n,$ there exists an $m\times m$-matrix $\mathbf{S}_i=(s^i_{jk})$ such that $\sum_{k=1}^ms^i_{jk}t'_{kl}=\delta_{ij}\delta_{il}$ for all $1\leq j,l\leq m.$ Take $s_{rl}^{r}$'s power of the equation \eqref{eq4.18} and we have
 \begin{equation*}
\prod_{j=1}^{m}\zeta_{m_i}^{\mathbbm{m}_i
x_{ij}s^r_{rl}t_{lj}}=\prod_{j=1}^{m}[\zeta_{m_i}^{a_is^r_{rl}t_{lj}s_{ji}\mathbbm{m}_i}\prod_{i<k\leq n}\zeta_{\mathbbm{m}_i\mathbbm{m}_k}^{a_{ik}s^r_{rl}t_{lj}s_{jk}\mathbbm{m}_i}].
\end{equation*}
By taking product of the identity above for $1\leq l\leq m,$ we get
 \begin{equation}\label{e4.17}
\prod_{l=1}^m\prod_{j=1}^{m}\zeta_{m_i}^{\mathbbm{m}_i
x_{ij}s^r_{rl}t_{lj}}=\prod_{l=1}^m\prod_{j=1}^{m}[\zeta_{m_i}^{a_is^r_{rl}t_{lj}s_{ji}\mathbbm{m}_i}\prod_{i<k\leq n}\zeta_{\mathbbm{m}_i\mathbbm{m}_k}^{a_{ik}s^r_{rl}t_{lj}s_{jk}\mathbbm{m}_i}].
\end{equation}
The left hand of the identity is $\prod_{l=1}^m\prod_{j=1}^{m}\zeta_{m_i}^{\mathbbm{m}_i
x_{ij}s^r_{rl}t_{lj}}=\prod_{l=1}^m\prod_{j=1}^{m}\zeta_{m_i}^{\mathbbm{m}_i
x_{ij}s^r_{rl}t'_{lj}}=\zeta_{m_i}^{\mathbbm{m}_ix_{ir}}.$
The right hand side is
\begin{eqnarray*}
&&\prod_{l=1}^m\prod_{j=1}^{m}[\zeta_{m_i}^{a_is^r_{rl}t_{lj}s_{ji}\mathbbm{m}_i}\prod_{i<k\leq n}\zeta_{\mathbbm{m}_i\mathbbm{m}_k}^{a_{ik}s^r_{rl}t_{lj}s_{jk}\mathbbm{m}_i}]\\
&=& \prod_{l=1}^m\prod_{j=1}^{m}[\zeta_{m_i}^{a_is^r_{rl}t'_{lj}s_{ji}\mathbbm{m}_i}\prod_{i<k\leq n}\zeta_{\mathbbm{m}_i\mathbbm{m}_k}^{a_{ik}s^r_{rl}t'_{lj}s_{jk}\mathbbm{m}_i}]\\
&=&\zeta_{m_i}^{a_is_{ri}\mathbbm{m}_i}\prod_{i<k\leq n}\zeta_{\mathbbm{m}_i\mathbbm{m}_k}^{a_{ik}s_{rk}\mathbbm{m}_i}.
\end{eqnarray*}
Hence \eqref{e4.17} is actually identical to \eqref{e4.10}.
\end{proof}

\begin{remark} It is well-known that any finite abelian group $\G$ can be written in the form $\G=\langle \g_1\rangle\times \cdots\times \langle \g_n\rangle$ such that the order of $\g_i$ is divided by that of $\g_j$ for $1\leq i<j\leq n.$ So we can assume that $\m_i|\m_j$ for $i<j$. \emph{We will keep this assumption in the rest of the paper.} In this way, the identities in \eqref{e4.9} equal to
$$\frac{\mathbbm{m}_l}{\mathbbm{m}_i}\sum_{j=1}^mx_{ij}t_{lj}\equiv a_{il}\; (\emph{mod} \;\mathbbm{m}_l), \;1\leq i<l\leq n.$$
\end{remark}

The above proposition implies that we don't have many choices on the sequence $\underline{\textbf{a}}\in A,$ see \eqref{3.1}.
\begin{corollary}\label{P4.9}
For the Nichols algebra $\B(V)'\in\ _G^G\mathcal{YD},$ there is at most one $\underline{\textbf{a}}\in A$ such that $\B(V)'^{J_{\underline{\mathbf{a}}}}$ is isomorphic to a Nichols algebra in $_\mathbbm{G}^\mathbbm{G}\mathcal{YD}^{\Phi_{\underline{\mathbf{a}}}}$ through $\pi$. Moreover, this $\underline{\textbf{a}}$ exists if and only if the Equations \eqref{e4.7} hold and in this case $\underline{\textbf{a}}$ can be taken in the following way:
\begin{equation}\label{eq4.19}
a_i\equiv \sum_{j=1}^mx_{ij}t_{ij}\;(\emph{mod}\; \mathbbm{m}_i); \;\;
a_{il}\equiv  \frac{\mathbbm{m}_l}{\mathbbm{m}_i}\sum_{j=1}^mx_{ij}t_{lj}\;(\emph{mod}\;\mathbbm{m}_l);\;\;
a_{ilt}=0
\end{equation}
for $1\leq i\leq n,\; 1\leq i<l\leq n$ and $1\leq i<l<t\leq n.$
\end{corollary}
\begin{proof}
By Proposition \ref{P4.8}, we know that $\underline{\textbf{a}}$ must satisfy Equations \eqref{e4.8} and \eqref{e4.9}. At the same time, since we always assume that $\Phi_{\underline{\textbf{a}}}$ is an abelian cocycle, $a_{rst}=0$ for all $1\leq r<s<t\leq n$ by  Proposition \ref{p3.9}. Therefore, there is at most one $\underline{\textbf{a}}$ satisfies these conditions. Proposition \ref{P4.8} also implies that such an $\underline{\textbf{a}}$ exists if and only if Equations \eqref{e4.7} hold.
\end{proof}

Now we are in the position to find the ``return trip" as follows.

\begin{figure}[hbt]
\begin{picture}(180,105)(0,0)
    \put(70,90){\makebox(0,0){$ \B(V)\in {^{\G}_{\G}\mathcal{YD}^{\Phi_{\underline{\textbf{a}}}}}$ }}
    \put(50,55){\vector(0,1){25}} \put(100,70){\makebox(0,0){\small\textsf{{Eq. \eqref{e4.7} satisfy}}}}
          \put(80,50){\makebox(0,0){$ \B(V)'^{J_{\underline{\textbf{a}}}}\in {^{G}_{G}\mathcal{YD}^{\pi^\ast(\Phi_{\underline{\textbf{a}}})}}$ }}
\put(110,30){\makebox(0,0){\small\textsf{ By Eq. \eqref{eq4.19}, find $\; J_{\underline{\textbf{a}}}$}}}  \put(50,15){\vector(0,1){25}}
     \put(100,10){\makebox(0,0){\small\textsf{Usual Nichols algebra }$\; \B(V)'\in {^{G}_{G}\mathcal{YD}}$}}

     \put(100,-10){\makebox(0,0){\small\textsf{Figure II}}}
\end{picture}
\end{figure}

\subsection{Root datum and Classification of Nichols algebras of diagonal type with arithmetic root systems in $_\mathbbm{G}^\mathbbm{G}\mathcal{YD}^\Phi$}

By Figure II, we want to formulate the conditions listed in Figure II by the language of root data. Using such language, we get a complete classification of Nichols algebras of diagonal type with arithmetic root systems in $_\mathbbm{G}^\mathbbm{G}\mathcal{YD}^\Phi$.

Now suppose $(\bigtriangleup,\chi,E)$ is an arithmetic root system,
$\mathcal{D}_{\chi,E}$ is the Dynkin diagram of $(\bigtriangleup,\chi,E).$ Up to twist equivalence, $(\bigtriangleup,\chi,E)$ is uniquely determined by $\mathcal{D}_{\chi,E}.$ In \cite{H3}, Heckenberger classified all the arithmetic root systems. Fix a Dynkin diagram with $m$ vertices, we call $$\{q_{ii}=\chi(e_i,e_i),\; \widetilde{q_{ij}}=\chi(e_i,e_j)\chi(e_j,e_i)|1\leq i,j\leq m\}$$ the structure constants of $\mathcal{D}_{\chi,E}.$

\begin{definition}
Let $\G=\Z_{\m_1}\times\cdots\times \Z_{\m_n}$ be the abelian group defined as above and set $m_i:=\m_i^2$ for $1\leq i\leq n$. Suppose $\mathcal{D}_{\chi,E}$ is a Dynkin diagram of an arithmetic root system $\triangle_{\chi,E}$ and $(q_{ii},\widetilde{q_{ij}})$ is the set of structure constants. If there exist parameter matrices $\mathbf{S}$ and $\mathbf{X}$ satisfying
\begin{itemize}
\item[1.]$\mathbf{S}=(s_{ij})_{m\times n}$ is a matrix with integer entries $0\leq s_{ij}<m_i$ for all $1\leq i\leq m, 1\leq j\leq n$ such that there exists a matrix $\mathbf{T}=(t_{ij})_{n\times m}$ satisfying \eqref{e4.6}.
\item[2.]$\mathbf{X}=(x_{ij})_{n\times m}$ with integer entries $0\leq x_{ij}<m_i$ for all $1\leq i\leq m, 1\leq j\leq n$ such that
 $q_{ii}=\prod_{k=1}^{n}\zeta_{m_k}^{s_{ik}x_{ki}},$  $\widetilde{ q_{ij}}=\prod_{k=1}^{n}\zeta_{m_k}^{s_{ik}x_{kj}+s_{jk}x_{ki}},$ and satisfy  Equations \eqref{e4.7}.
\end{itemize}
Then we call $\mathfrak{D}=\mathfrak{D}(\mathcal{D}_{\chi,E},\mathbf{S},\mathbf{X})$ a root datum over $\G$ and moreover we call $\triangle$ (resp. $\triangle_{\chi,E}$) the root system (resp. arithmetic root system) of $\mathfrak{D}$.
\end{definition}

For a fixed root datum $\mathfrak{D}=\mathfrak{D}(\mathcal{D}_{\chi,E},\mathbf{S},\mathbf{X})$ over $\G$, define a sequence $\underline{\textbf{a}}\in A$ through Equations \eqref{eq4.19}.
Now we can define a Nichols algebra $\B(V_\mathfrak{D})\in {_G^{G}\mathcal{YD}^{\pi^{\ast}(\Phi_{\underline{\mathbf{a}}})}}$ in the following way: Let $V_{\mathfrak{D}}$ be the Yetter-Drinfeld module in $ _G^G\mathcal{YD}^{\pi^*(\Phi_{\underline{\mathbf{a}}})}$ with a canonical basis $\{X_i|1\leq i\leq m\}$ such that $$\delta_L(X_i)=\prod_{k=1}^{n}g_k^{s_{ik}}\otimes X_i,\ \ g_i\blacktriangleright X_j=\zeta_{m_i}^{x_{ij}}\frac{J_{\underline{\textbf{a}}}
(g_i,\prod_{k=1}^{n}g_k^{s_{ik}})}{J_{\underline{\textbf{a}}}(\prod_{k=1}^{n}g_k^{s_{ik}},g_i)}X_j.$$
Now the main result of this section can be stated as follows.
\begin{theorem}\label{T4.10}
\begin{itemize}
\item[(1)] The Nichols algebra $\B(V_\mathfrak{D})$ is isomorphic to a Nichols algebra of diagonal type with arithmetic root system in $_\mathbbm{G}^\mathbbm{G}\mathcal{YD}^{\Phi_{\underline{\mathbf{a}}}}$ through the group epimorphism $\pi: G\to \G.$
\item[(2)]Suppose $\B(V)$ is a Nichols algebra of diagonal type with arithmetic root system in $_\mathbbm{G}^\mathbbm{G}\mathcal{YD}^{\Phi_{\underline{\mathbf{a}}}}$ and the support group is $\mathbbm{G},$ then there exists a root datum $\mathfrak{D}$ over $\G$ such that $\B(V_\mathfrak{D})\cong \B(V)$ through the group epimorphism $\pi: G\to \G.$
\end{itemize}
\end{theorem}
\begin{proof} The first statement is just a direct consequence of Figure II and the definition of a root datum. Now we show the second statement. According to Figure I, from $\B(V)$ one can construct a usual Nichols algebra $\B(V)'$. By the construction of $\B(V)'$ we can find that $\B(V)'^{J_{\underline{\textbf{a}}}}$ is isomorphic to $\B(V)$ through $\pi$. By Proposition \ref{P4.8} and the definition of a root datum, we know that there is a root datum $\mathfrak{D}$ over $\G$ such that $\B(V)'^{J_{\underline{\textbf{a}}}}=\B(V_\mathfrak{D})$.
\end{proof}

\begin{convention} \emph{By this theorem, we know that the Nichols algebra $\B(V_\mathfrak{D})$ is isomorphic to a unique Nichols algebra in $_\mathbbm{G}^\mathbbm{G}\mathcal{YD}^{\Phi_{\underline{\mathbf{a}}}}$. For convenience, this Nichols algebra is denoted by $\B(\mathfrak{D})$.}
\end{convention}

\section{Classification results}
In this section, all the finite-dimensional connected graded pointed Majid algebras $\MM$ of diagonal type will be classified. The main idea is to show that the coninvariant subalgebra of $\MM$ is indeed a Nichols algebra of diagonal type and from this we can apply the classification results gotten in the previous section.
\subsection{General setup.} In this section, we always assume that $\MM$ is a finite-dimensional connected coradically graded pointed Majid algebra of diagonal type.
 From Subsection 2.1, we know that $$\MM_{0}=(\k\mathbb{G},\Phi)$$
 where $\mathbbm{G}$ is the group consisting of all the group-like elements and $\Phi$ is a $3$-cocycle on $\mathbbm{G}.$ Using the same arguments given in Proposition \ref{P4.7} and our assumption that $\MM$ being connected, we know that $\G$ is abelian and $\Phi$ is an abelian cocycle. Therefore, $$\mathbb{G}=\Z_{\mathbbm{m}_1}\times \cdots\times \Z_{\mathbbm{m}_n} = \langle \mathbbm{g}_1 \rangle \times \cdots \times \langle \mathbbm{g}_n \rangle $$
with $\mathbbm{m}_i|\mathbbm{m}_j$ for $1\leq i<j\leq n,$ and
$$\Phi=\Phi_{\underline{\mathbf{a}}}$$
for some $\underline{\mathbf{a}}\in A$ with $a_{rst}=0$ for all $1\leq r<s<t\leq n$.

Let $\R$ be the coinvariant subalgebra of $\MM,$ then $\R$ is a Hopf algebra in $_\mathbbm{G}^\mathbbm{G}\mathcal{YD}^\Phi$ and
 $$\MM=\R\# \k\G.$$
 The main task of this section is to show that $\R$ is indeed a Nichols algebra in $_\mathbbm{G}^\mathbbm{G}\mathcal{YD}^\Phi$. From the classification results obtained in the previous section, we can classify $\MM$ directly.


\subsection{$\R$ is a Nichols algebra}

 Note that we already showed that each finite-dimensional rank $2$ pointed Majid algebra is generated by group-like and skew-primitive elements in \cite{qha6}. We gradually realize that the methods developed in \cite{qha6} still work for pointed Majid algebras of diagonal type. For completeness and safety, the proof will be given though it is similar to the version of \emph{ibid.}.

 The main result of this subsection can be stated as follows.

 \begin{proposition}\label{P5.3} In ${^{\mathbb{G}}_{\mathbb{G}}\mathcal{YD}^{\Phi_{\underline{\textbf{a}}}}}$, we have $$\R\cong \B(\R_1).$$
 \end{proposition}

%
%
%
We give several preparations. Take a Nichols algebra $\B(V)$ of diagonal type in ${^{\G}_{\G}}\mathcal{YD}^{\Phi_{\underline{\textbf{a}}}}$. Then according to the Figure I, we have $\B(V)\cong \B(\widetilde{V})\in {^{G}_{G}}\mathcal{YD}^{\pi^\ast(\Phi_{\underline{\textbf{a}}})}$ and $\B(\widetilde{V})^{J_{\underline{\textbf{a}}}^{-1}}=
\B(\widetilde{V}^{J_{\underline{\textbf{a}}}^{-1}})$ is a usual Nichols algebra in $^G_{G}\mathcal{YD}$. As before, we denote this usual Nichols algebra by $\B(V)'$ and we use $V'$ to denote the Yetter-Drinfeld module $\widetilde{V}^{J_{\underline{\textbf{a}}}^{-1}}$. That is, $\B(V)'=\B(V')$. Let $\{X_i|1\leq i\leq m\}$ be a canonical basis of $V'$. Then $\B(V')\cong T(V')/I$ where $I$ is the Hopf ideal of $T(V')$ generated by the polynomials in $\{X_i|1\leq i\leq m\}$ listed in \cite[Theorem 3.1]{Ang}. In the following, let $\mathbf{S}$ denote the set of these polynomials. Define a map $\Psi:  T_{\partial(J_{\underline{\textbf{a}}})}(\widetilde{V})=
T_{\partial(J_{\underline{\textbf{a}}})}(V'^{J_{\underline{\textbf{a}}}})\to T(V')$ by
\begin{equation} \label{ne1}
\Psi((\cdots ((Y_1\circ Y_2)\circ Y_3)\cdots Y_n))=\prod_{i=1}^{n-1}J_{\underline{\textbf{a}}}(y_1\cdots y_i,y_{i+1})Y_1Y_2\cdots Y_n
\end{equation} for all $Y_i\in \{X_1,X_2,\ldots,X_m\}.$ It is easy to see that $\Psi$ is an isomorphism of linear spaces. The following conclusion is  \cite[Lemma 4.5]{qha6}.

\begin{lemma}\label{l4.15}
The set $\Psi^{-1}(\mathbf{S})$ is a minimal set of defining relations of $\B(V')^{J_{\underline{\textbf{a}}}}=\B(\widetilde{V}).$
\end{lemma}
%
%
We also need the following two lemmas, which were given essentially in \cite{Ang} and were rephrased as follows in \cite{qha6}.
\begin{lemma}\cite[Lemma 4.6]{qha6}\label{l4.16}
Let $Z$ be a polynomial in $\mathbf{S},$ then we have $\B(V'^{J_{\underline{\textbf{a}}}}\oplus \k\Psi^{-1}(Z))\cong \B(V'\oplus \k Z)^{J_{\underline{\textbf{a}}}}.$
\end{lemma}

\begin{lemma}\cite[Lemma 4.7]{qha6}\label{l4.17}
Let $\B(V')\in \ _G^G\mathcal{YD}$ be a finite-dimensional Nichols algebra of diagonal type. Suppose that $Z$ is a polynomial in $\mathbf{S}$ and $U'=V'\oplus \k Z,$ then $\B(U')$ is infinite-dimensional.
\end{lemma}

The following is a generalized version of \cite[Propostion 4.8]{qha6}, where we proved it in rank two case.
\begin{proposition}\label{p3.5}
Let $R=\oplus_{i\geq 0} R_i$ be a finite-dimensional graded (not necessarily coradically graded) Hopf algebra in ${_\mathbbm{G}^\mathbbm{G}\mathcal{YD}^\Phi}$. If $R$ is generated by $R_1,$ then $R=\B(R_1).$
\end{proposition}
\begin{proof}
Let $I$ be an ideal of $T_{\Phi}(R_1)$ such that $R=T_{\Phi}(R_1)/I.$ Clearly, we have a surjective Hopf map $$\theta:\; R\twoheadrightarrow \B(R_1).$$  By Proposition \ref{P4.5}, $\B(R_1)$ is also a Nichols algebra in ${^G_G\mathcal{YD}^{\pi^{*}(\Phi)}}$ for $G=\Z_{m_1}\times\cdots\times  \Z_{m_n}$. By Proposition \ref{p3.10}, $\pi^{*}(\Phi)=\partial{(J)}$. Therefore, $\B{(R_1)}^{J^{-1}} \in {^G_G\mathcal{YD}}$ is a usual Nichols algebra. Now assume that $\theta$ is not an isomorphism, then there should be some polynomials in $\Psi^{-1}(\mathbf{S}),$ which are not contained in $I$ by Lemma \ref{l4.15}. Suppose that $\Psi^{-1}(Z)$ is one of those with minimal length. Then we know that $\Psi^{-1}(Z)$ must be a primitive element in $R.$ Let $U=R_1\oplus \k\Psi^{-1}(Z),$ then by the preceding assumption there is an embedding of linear spaces $\B(U)\hookrightarrow R.$

 We already know that $\B(R_1)^{J^{-1}}$ is a finite-dimensional Nichols algebra in $_G^G\mathcal{YD}$. It is not hard to see that there exists $R_1'\in {_G^G\mathcal{YD}}$ such that $R_1=R_1'^J$ (since $J$ induces an equivalence between $_G^G\mathcal{YD}$ and $_G^G\mathcal{YD}^{\partial (J)}={_G^G\mathcal{YD}^{\pi^{\ast}(\Phi)}}$).  By Lemma \ref{l4.16}, we have $\B(R_1'\oplus Z)^J=\B(R_1\oplus \k\Psi^{-1}(Z))=\B(U).$ Note that $\B(R_1'\oplus Z)^J$ is infinite-dimensional due to Lemma \ref{l4.17}. Hence $\B(U)$ is infinite-dimensional, which contradicts to the assumption that $R$ is finite-dimensional. Thus $\theta$ is an isomorphism and $R$ is the Nichols algebra $\B(R_1).$
\end{proof}
In order to prove Proposition \ref{P5.3}, we still need the following lemma.

\begin{lemma}\label{l5.7}
Let $R=\oplus_{i\geq 0}R_i$ be a graded Hopf algebra in $_\mathbbm{G}^\mathbbm{G} \mathcal{YD}^\Phi$ with $R_0=\k 1$ and $P(R)=R_{1}.$ Then the right dual $R^*=\oplus_{i\geq 0}R_i^*$  (resp. the left dual $^*R=\oplus_{i\geq 0}{^*R_i}$) is generated by $R_{1}^*$  $(resp. \ {^*R_{1}}).$
\end{lemma}
\begin{proof}
Note that the proof of \cite[Lemma 4.10]{qha6} does not depend on the abelian group $G$ and abelian $3$-cocycle $\Phi$ of $G.$ Hence we can prove the lemma in the same way.
\end{proof}

At last, we can prove Proposition \ref{P5.3}.

\noindent \textbf{Proof of Proposition \ref{P5.3}.} By assumption, $\R_0=\k 1$ and $P(\R)=\R_1.$ According to Lemma \ref{l5.7}, $\R^*=\oplus_{i\geq 0}\R_i^*$ is generated by $\R_{1}^*.$ By Proposition \ref{p3.5}, $\R^*=\B(\R_{1}^*).$ So we have $P(\R^*)=\R_{1}^*,$ and  $^*(\R^*)=\R$ is generated by $\R_{1}$ according to Lemma \ref{l5.7} again. Hence $\R$ is also a Nichols algebra by Proposition \ref{p3.5}. Thus, $\R=\B(\R_{1})$.

\subsection{Classification result}
 For a root datum $\mathfrak{D}=\mathfrak{D}(\mathcal{D}_{\chi,E},\mathbf{S},\mathbf{X}),$ there is a Nichols algebra $\B(\mathfrak{D})$ in  $_\mathbbm{G}^\mathbbm{G}\mathcal{YD}^{\Phi_{\underline{\mathbf{a}}}},$ where $\underline{\textbf{a}}\in A$ is determined by Equations \eqref{eq4.19}. Denote $\MM (\mathfrak{D})=\B(\mathfrak{D})\#\k \mathbbm{G},$ then we can formulate the main result of the paper.

\begin{theorem}\label{T5.8} Keep the notations as before. We have
\begin{itemize}
\item[(1).] The Majid algebra $\MM(\mathfrak{D})$ is a connected coradical graded pointed Majid algebra of diagonal type over  the group $\mathbbm{G}$. Moreover, $\MM(\mathfrak{D})$ is finite-dimensional if and only if the heights of all restricted Poincare-Birkhoff-Witt generators of $\B(\mathfrak{D})$ are finite.
\item[(2).] Any finite-dimensional connected coradical graded pointed Majid algebra of diagonal type over $\mathbbm{G}$ is isomorphic to a $\MM(\mathfrak{D})$ for some $\mathfrak{D}.$
    \end{itemize}
\end{theorem}
\begin{proof}
 It follows from  Proposition \ref{P5.3} and Theorem \ref{T4.10}.
\end{proof}

\subsection{A corollary}
In \cite[Conjecture 1.4]{as}, Andruskiewitch-Schneider Conjectured that every finite-dimensional pointed Hopf algebras over $\k$ is generated by groups-like and skew-primitive elements. This is the so called generation problem, which plays an important role on the classification of pointed Hopf algebras. It is true in many cases, see \cite{Ang}. This conjecture was generalized to finite-dimensional pointed Majid algebras or even to pointed finite tensor categories \cite{EGNO}.

\begin{corollary}
Suppose $\MM$ is a finite-dimensional pointed Majid algebras of diagonal type, then $\MM$ is generated by group-like and skew-primitive elements.
\end{corollary}
\begin{proof}
Since $\MM$ is generated by group-like and skew-primitive elements if and only if its coradically graded version $\gr (\MM)$ is, we can assume that $\MM$ is coradically graded. Let $\R$ be the coinvariant subalgebra of $\MM,$ and assume that its support group is $H.$ Then $\R\# \k H$ is a finite-dimensional connected coradically graded pointed Majid algebra of diagonal type, and thus it is generated by group-like and skew-primitive elements according to Theorem \ref{T5.8}. This implies that $\MM$ is also generated by group-like and skew-primitive elements.
\end{proof}

\section{Examples of genuine pointed Majid algebras}
 In this section, we provide some methods to construct genuine pointed Majid algebras from arithmetic root systems.  For each arithmetic root system $\bigtriangleup_{\chi,E}$ of rank $\theta$ satisfying a mild condition, we show that there always exists a genuine pointed Majid algebra of standard type $\MM\cong \B(V)\#\k \Z_m^\theta$, such that the arithmetic root system of $\B(V)$ is $\bigtriangleup_{\chi,E}$.  For arithmetic root systems of Cartan type, we also provide a unified method to construct genuine finite-dimensional pointed Majid algebras.

\subsection{Pointed Majid algebras of standard type}
Suppose that $\MM$ is a pointed Majid algebra generated by the abelian group $\mathbbm{G}$ and skew-primitive elements $\{X_1,\cdots,X_n\}$ with $\Delta(X_i)=X_i\otimes 1+\mathbbm{g}_i\otimes X_i, 1\leq i\leq n$. Then as our previous paper \cite{qha6}, we call $\MM$ is of \emph{standard} type if $\mathbbm{G}=\langle \mathbbm{g}_1\rangle \times \cdots\times \langle \mathbbm{g}_n\rangle.$ This definition is transferred naturally to Yetter-Drinfeld modules and thus Nichols algebras.

\begin{definition}
Let $V$ be a Yetter-Drinfeld module in $_\mathbbm{G}^\mathbbm{G}\mathcal{YD}^{\Phi_{\underline{\textbf{a}}}}$ of diagonal type and $V$
is said to be of standard type if there exists a canonical basis $X_1,\cdots,X_n$ with degrees $\mathbbm{g}_1,\cdots, \mathbbm{g}_n$ respectively such that  $\mathbbm{G}=\langle \mathbbm{g}_1\rangle \times \cdots\times \langle \mathbbm{g}_n\rangle.$ The Nichols algebra $\B(V)\in \ _\mathbbm{G}^\mathbbm{G}\mathcal{YD}^{\Phi_{\underline{\textbf{a}}}}$ is standard if $V$ is so.
\end{definition}
The definition is independent of the choice of canonical basis. If $V\in{_\mathbbm{G}^\mathbbm{G}\mathcal{YD}^{\Phi_{\underline{\textbf{a}}}}}$ is a Yetter-Drinfeld module of standard type, then $\B(V)\#\k \mathbbm{G}$ is a pointed Majid algebra of standard type.

\begin{lemma}\label{P6.2} Let $\mathbbm{G}=\langle \mathbbm{g}_1\rangle \times \cdots\times \langle \mathbbm{g}_n\rangle=\Z_{\m_1}\times\cdots\times\Z_{\m_n}$ and $\Phi_{\underline{\textbf{a}}}$ an abelian $3$-cocycle on $\mathbbm{G}$.
 \begin{itemize}
 \item[(1).] Suppose $V$ is a Yetter-Drinfeld module of standard type in $_\mathbbm{G}^\mathbbm{G}\mathcal{YD}^{\Phi_{\underline{\textbf{a}}}}$ and $X_1,\cdots,X_n $ is a canonical basis of $V$. Let $(x_{ij})$ be the numbers satisfying  $\mathbbm{g}_i\triangleright X_j=\zeta_{\mathbbm{m}_i^2}^{x_{ij}}X_j$ for all $1\leq i,j\leq n$. Then we have
\begin{equation}\label{eq6.1}
x_{ij}\equiv 0\;(\emph{mod}\;\mathbbm{m}_i)\;(i>j)\end{equation}
\begin{equation}\label{eq6.3}
 x_{ii}\equiv a_i\;(\emph{mod}\; \mathbbm{m}_i);\;\;\;\;
\mathbbm{m}_jx_{ij}\equiv \mathbbm{m}_ia_{ij}\;(\emph{mod}\; \mathbbm{m}_i\mathbbm{m}_j)\; (i<j).
\end{equation}
\item[(2).] Conversely, let $V=\k\{X_1,\cdots ,X_n\}$ be a $\k\mathbbm{G}$-comodule such that the degree of $X_i$ is $\mathbbm{g}_i$ for $1\leq i\leq n$. If we have numbers $(x_{ij})_{1\leq i,j\leq n}$ satisfying equations \eqref{eq6.1}, then the action $$\mathbbm{g}_i\triangleright X_j:=\zeta_{\mathbbm{m}_i^2}^{x_{ij}}X_j,\;(1\leq i,j\leq n)$$ makes $V$  a standard  Yetter-Drinfeld module in $_\mathbbm{G}^\mathbbm{G}\mathcal{YD}^{\Phi_{\underline{\textbf{a}}}}$, where the sequence $\underline{\textbf{a}}$ are the numbers determined by the Equations \eqref{eq6.3}.
\end{itemize}
\end{lemma}
\begin{proof}
 (1). Let $G=\langle g_1\rangle \times \cdots\times \langle g_n\rangle$ such that $|g_i|=\mathbbm{m}_i^2$, $1\leq i\leq n$. By Proposition \ref{P4.5}, $\B(V)$ can be viewed as a Yetter-Drinfeld module in $_G^G\mathcal{YD}^{\pi^*({\Phi_{\underline{\textbf{a}}}})}$ through $g_i\triangleright X_j=\zeta_{\mathbbm{m}_i^2}^{x_{ij}}X_j$, $1\leq i,j\leq n$. So $\B(V)^{J^{-1}_{\underline{a}}}\in {_G^G\mathcal{YD}}$. Since $g_i\triangleright_{J^{-1}_{\underline{a}}} X_j=\frac{J(g_i,g_j)}{J(g_j,g_i)}g_i\triangleright X$ and $J(g_i,g_j)=J(g_j,g_i)=1 $, we have $g_i\triangleright_{J^{-1}_{\underline{\textbf{a}}}} X_j=\zeta_{\mathbbm{m}_i^2}^{x_{ij}}X_j$. Identities \eqref{eq6.1}-\eqref{eq6.3} follow from Proposition \ref{P4.8}.

 (2). Follows from Proposition \ref{P4.8}.
\end{proof}

In the following, for a root of unity $q$ we use $|q|$ to denote its order.

\begin{proposition}\label{P6.3}
Let $\bigtriangleup_{\chi,E}$ be a connected arithmetic root system of rank $\theta$ listed in \cite{H3} and let $\{q_{ii},\widetilde{q_{ij}}|1\leq i<j\leq \theta\}$ be its structure constants.
\begin{itemize}
\item[(a).]
If there is a $q_{ii}$ or $\widetilde{q_{ij}}$ such that its order is not of the form $p_1p_2\cdots p_n$, where $p_1\neq p_2\neq \cdots \neq p_n$ are prime numbers, then there is a standard Yetter-Drinfeld module  $V\in{_{\k\Z_m^\theta}^{\k\Z_m^\theta}\mathcal{YD}^{\Phi_{\underline{\textbf{a}}}}}$ for some $m$ and $\Phi_{\underline{\textbf{a}}}$ such that $\B(V)\#\k \Z_m^\theta$ is a genuine pointed Majid algebra of standard type, and the arithmetic root system of $\B(V)$ is equal to $\triangle_{\chi,E}.$
\item[(b).]
 If the order of each $q_{ii}$ and $\widetilde{q_{ij}}$ is of the form $p_1p_2\cdots p_n$ for some distinct prime numbers $p_1,p_1,\cdots, p_n,$ then there is no standard Yetter-Drinfeld modules in $_{\k\Z_m^\theta}^{\k\Z_m^\theta}\mathcal{YD}^{\Phi_{\underline{\textbf{a}}}}$ for all $m\geq 1$ and $\Phi_{\underline{\textbf{a}}}\neq 1$.
 \end{itemize}
\end{proposition}
\begin{proof}
(a). Firstly, we define a function $\upsilon:\N\to \N$ by
\begin{equation}
\upsilon(k)=\left\{
         \begin{array}{ll}
           k+1, & \hbox{if $k$ is odd;} \\
           k, & \hbox{if $k$ is even.}
           \end{array}
       \right.
\end{equation}
 Let $\Upsilon$ be the map
 \begin{eqnarray}
 \Upsilon:\N&\longrightarrow & \N \\
k=p_1^{N_1}p_2^{N_2}\cdots p_n^{N_n}&\mapsto&\Upsilon(k)=p_1^{\upsilon(N_1)}p_2^{\upsilon(N_2)}\cdots p_n^{\upsilon(N_n)}.
 \end{eqnarray}
Here $p_1,p_2,\cdots p_n$ are mutually distinct prime numbers.
Let $m_i=|q_{ii}|, m_{ij}=|\widetilde{q_{ij}}|$ for $1\leq i,j\leq \theta$. By $\mathbbm{m}$ we denote the minimal positive number such that $\Upsilon(m_i)|\mathbbm{m}$ and $\Upsilon(m_{ij})|\mathbbm{m}$ for all $1\leq i,j\leq \theta$. Then it is obvious that $\sqrt{\mathbbm{m}}$ is a positive integer, since all $\sqrt{\Upsilon(m_i)}, \sqrt{\Upsilon(m_{ij})}, 1\leq i,j\leq \theta$ are integers. Let $m=\sqrt{\mathbbm{m}}$.

 Next we will show that there is a standard Yetter-Drinfeld module $V\in{_{\k\Z_m^\theta}^{\k\Z_m^\theta}\mathcal{YD}^{\Phi_{\underline{\textbf{a}}}}}$ for  some  nontrivial $3$-cocycle $\Phi_{\underline{\textbf{a}}}$ on $\Z_m^\theta$. Suppose $\mathbbm{g}_1,\cdots,\mathbbm{g}_\theta$ are free generators of $\Z_m^\theta$, that is $\Z_m^\theta=\langle \g_1\rangle\times \cdots\times \langle \g_\theta\rangle$, and $V=\k\{X_1,\cdots,X_\theta\}$ a $\Z_m^\theta$-graded vector space. For $1\leq i,j\leq \theta$,  define
 \begin{equation}
x_{ij}=\left\{
         \begin{array}{ll}
         \frac{m^2}{m_i}, & \hbox{if  $i=j$ ;} \\
           0, & \hbox{if $\widetilde{q_{ij}}= 1$, $i\neq j$ ;} \\
           0, & \hbox{if $\widetilde{q_{ij}}\neq 1$ and $i>j$;} \\
           \frac{m^2}{m_{ij}}, & \hbox{if $\widetilde{q_{ij}}\neq 1$ and $i<j$ .}
           \end{array}
       \right.
\end{equation}
According to (2) of Lemma \ref{P6.2}, $\mathbbm{g}_i\triangleright X_j=\zeta_{m^2}^{x_{ij}}X_j$, $1\leq i,j\leq \theta$ makes $V$  a standard Yetter-Drinfeld module in  $_{\k \Z_m^\theta}^{\k \Z_m^\theta}\mathcal{YD}^{\Phi_{\underline{\textbf{a}}}}$, where the sequence $\underline{\textbf{a}}$ is determined by Equations \eqref{eq6.3}.
Since $\zeta_{m^2}^{x_{ii}}=\zeta_{m^2}^{\frac{m^2}{m_i}}=q_{ii}$, $\zeta_{m^2}^{x_{ij}}\zeta_{m^2}^{x_{ji}}=\zeta_{m^2}^{\frac{m^2}{m_{ij}}}=\widetilde{q_{ij}}$ for $1\leq i,j\leq \theta$, we prove that the arithmetic root system of $\B(V)$ is equal to $\triangle_{\chi,E}$.

At last, we will show that $\underline{\textbf{a}}$ is nonzero. From the assumption of the first part of the proposition,  there is an element $\zeta$ in $\{q_{ii},\widetilde{q_{ij}}\}$ satisfying the following conditions:
\begin{itemize}
\item[C.1] $|\zeta|=p_1^{N_1}p_2^{N_2}\cdots p^{N_n}_n$ and there exists some $l$ such that $N_l\geq 2.$ Here $p_1,p_2,\cdots,p_n$ are mutually distinct prime numbers.
\item[C.2] $p_l^{N_l+1}\nmid m_i,\ \ p_l^{N_l+1}\nmid m_{ij}$ if $m_i\neq |\zeta|, m_{ij}\neq |\zeta|$.
\end{itemize}
If $\zeta=q_{ii}$ for some $i$, then by the definition of $m$ and the choice of $\zeta$, we have $m\nmid\frac{m^2}{m_i}$, which implies $a_i\neq 0 \;(\textrm{mod}\; m)$ by Equations \eqref{eq6.3}. Similarly, if $\zeta=\widetilde{q_{ij}}$ for some $i,j$, then one can prove that $a_{ij}\neq 0 \;(\textrm{mod}\; m)$ by \eqref{eq6.3}. We have proved that $\underline{\textbf{a}}$ is nonzero.

(b). Suppose there is a standard Yetter-Drinfeld module $V$ in $_{\k\Z_m^\theta}^{\k\Z_m^\theta}\mathcal{YD}^{\Phi_{\underline{\textbf{a}}}}$ for some $m\geq 1$ and  $\Phi_{\underline{\textbf{a}}}$. Then we will prove that $\Phi_{\underline{\textbf{a}}}=1.$

On the one hand, let $\{X_1,\cdots,X_\theta\}$ be a canonical basis of $V$ and $\{\mathbbm{g}_1,\cdots,\mathbbm{g}_\theta\}$ the corresponding degrees. Since $V$ is standard Yetter-Drinfeld module,  $\Z_m^\theta=\langle \mathbbm{g}_1\rangle\times \cdots \times \langle \mathbbm{g}_\theta\rangle$. Let $(x_{ij})_{1\leq i,j\leq \theta}$ be the numbers defined by $\mathbbm{g}_i\triangleright X_j=\zeta_{m^2}^{x_{ij}}X_j, \; 0\leq x_{ij}<m^2.$ So by  Equations \eqref{eq6.3}, we have \begin{eqnarray}
&& x_{ii}\equiv a_i\ (\textrm{mod} \; m),\ 1\leq i\leq \theta;\label{eq6.8}\\
&&x_{il}\equiv a_{il}\ (\textrm{mod}\; m),\ 1\leq i<l\leq \theta.\label{eq6.9}
\end{eqnarray}

On the other hand, since the order of $\zeta_{m^2}^{x_{ii}}=q_{ii}$ is of the form $p_1\cdots p_n$, where $p_1,\cdots,p_n$ are mutually distinct prime numbers, we have $p_1\cdots p_n|m$, and hence $m|x_{ii}$. This implies that $a_i=0$ by Equations \eqref{eq6.8}. Similarly, from  Equations \eqref{eq6.9} one can show that $a_{ij}=0, \leq i<j\leq \theta.$
\end{proof}

According to this proposition, we can construct a big class of genuine pointed Majid algebras such that the corresponding Nichols algebras have arithmetic root systems.

\begin{example}\label{E6.4}
 Let $\bigtriangleup_{\chi,E}$ be an arithmetic root system of the following type:
\begin{itemize}
\item[1.] Rank $2$ arithmetic root systems of cases 1-5, 7-12, 14 as listed in \cite[Table 1]{H3},\\
\item[2.] Rank $3$ arithmetic root systems of cases 1-8, 10, 18 as listed in \cite[Table 2]{H3}, \\
\item[3.] Rank $4$ arithmetic root systems of cases 1-14, 22 as listed in \cite[Table 3]{H3},\\
\item[4.] Higher rank ($\geq 5$) arithmetic root systems of cases 1-4, 7-10, 14, 19, 22 as listed in \cite[Table 4]{H3},
\end{itemize}
such that the parameter $q$ (if there is a parameter $q$ in the root system) is a root of unity, and the order of $q$ is of the form $p_1^{N_1}p_2^{N_2}\cdots p_n^{N_n}$, where $p_1,p_2,\cdots,p_n$ are mutually distinct prime numbers, $n\geq 3$ and there exists at least one $N_i \ge 2$ for some $1\leq i\leq n$. The we can construct a genuine pointed Majid algebras of standard type $\MM=\B(V)\# \k \Z_m^\theta$ such that the arithmetic root system of $\B(V)$ is $\triangle_{\chi,E},$
where the number $m$ is listed in Table 6.1.

{\setlength{\unitlength}{1mm}
\begin{table}[t]\centering
\caption{The number $m$ associated to each arithmetic root system} \label{tab.1}
\vspace{1mm}
\begin{tabular}{r|l|l}
\hline
  &\text{Arithmetic root systems}& \text{$m$}\\
\hline
\hline
  1.  & Rank $2$: cases 1-4, 10; rank $3$: cases 1-8, 10;  & $m=\Upsilon(|q|)$, if $2\mid|q|$;\\
  &  rank $4$: cases 1-14, rank $\geq 5$: cases 1-4, 7-10, 22 &$m=2\Upsilon(|q|)$, if $2\nmid|q|$\\
\hline
  2. & Rank $2$: cases 5&   $m=\Upsilon(|q|)$, if $3\mid|q|$;  \\
  &   &    $m=3\Upsilon(|q|)$, if $3\nmid|q|$ \\
  \hline
  3. &Rank $2$: cases 7-9& \ \ \ \ \  $m=6$\\
  \hline
  4. &Rank $2$: cases 11& \ \ \ \ \  $m=4$\\
  \hline
  5. &Rank $2$: cases 12& \ \ \ \ \  $m=12$\\
  \hline
  6. &Rank $2$: cases 14& \ \ \ \ \  $m=10$\\
  \hline
  7. &Rank $3$: cases 18& \ \ \ \ \  $m=3$\\
  \hline
  8. &Rank $4$: cases 22, Rank $5$: cases 14,19 & \ \ \ \ \  $m=2$\\
  \hline
 \end{tabular}
\end{table}}
Explicitly, let $\Z_m^\theta=\langle \mathbbm{g}_1\rangle \times \cdots \times \langle \mathbbm{g}_\theta\rangle$ and $V=\k \{X_1,\cdots, X_\theta\}$ be a $\Z_m^\theta$-graded vector space with $\deg X_i=\g_{i}$ for $1\leq i\leq \theta$. Define
   \begin{equation}
\mathbbm{g}_i\triangleright X_j=\left\{
         \begin{array}{ll}
         q_{ii}X_i, & \hbox{if  $i=j$ ;} \\
          \widetilde{q_{ij}}X_j, & \hbox{if $i<j$ ;} \\
          X_j, & \hbox{if $i>j$ .}
           \end{array}
       \right.
\end{equation}
Then $V\in {_{\k\Z_m^\theta}^{\k\Z_m^\theta}\mathcal{YD}^{\Phi_{\underline{\textbf{a}}}}}$, $\underline{\textbf{a}}$ is determined by Equations \eqref{eq6.3}. According to Proposition \ref{P6.3}, $\Phi_{\underline{\textbf{a}}}$ is nontrivial and $\MM=\B(V)\#\k \Z_m^\theta$ is a genuine pointed Majid algebra such that the arithmetic root system of $\B(V)$ is $\bigtriangleup_{\chi,E}$.
\end{example}

\begin{remark}
The preceding construction provides many new examples of finite-dimensional pointed Majid algebras. It is obvious that $\MM=\B(V)\#\k \Z_m^\theta$ is finite-dimensional if and only if $\B(V)$ is finite-dimensional, which is completely determined by its arithmetic root system ($=\bigtriangleup_{\chi,E}$). When $\bigtriangleup_{\chi,E}$ is of rank $2$, or of Cartan type, the the associated pointed Majid algebra $\MM=\B(V)\#\k \Z_m^\theta$ is finite-dimensional. For other cases, it is an open question whether the corresponding Nichols algebras is finite-dimesional or not.
\end{remark}

\subsection{Finite-dimensional pointed Majid algebras of Cartan type} In this subsection, we will give more examples of genuine finite-dimensional pointed Majid algebras.
Let $\bigtriangleup_{\chi,E}$ be an arithmetic root system. If $\bigtriangleup$ is a root system of a complex semisimple Lie algebra, then we call $\bigtriangleup_{\chi,E}$ an arithmetic root system of Cartan type.

\begin{definition}
Let $\MM$ be a finite-dimensional connected graded pointed Majid algebra. By Theorem \ref{T5.8}, there exists a root datum $\mathfrak{D}=\mathfrak{D}(\mathcal{D}_{\chi,E},\mathbf{S},\mathbf{X})$ such that $\MM=\MM(\mathfrak{D})$. If the arithmetic root system of $\mathfrak{D}$ is of Cartan type, then we say that $\MM$ is a pointed Majid algebra of Cartan type.
\end{definition}

According to \cite{H3}, if $\B(V)$ is a usual Nichols algebra of diagonal type with finite root system, and the rank of $\B(V)$ is $\theta,$ then there exist a bicharacter $\chi$ on $\Z^\theta$ and a basis $E$ such that $\bigtriangleup(\B(V))_{\chi,E}$ is an arithmetical root system. In fact, arithmetic root systems include more information than root systems of Nichlos algebras. For instance, all the rank $1$ Nichols algebras have the same root system, i.e.,
$\{\alpha,-\alpha\}$, but there are both finite-dimensional and infinite-dimensional rank $1$ Nichols algebras, hence have different arithmetic root systems.

Suppose $\mathrm{C}=(c_{ij})_{1\leq i,j\leq \theta}$ is a finite Cartan matrix and we say $1\leq i\neq j\leq \theta$ are connected if there exist $k_1,k_2,\cdots,k_n$ such that $c_{i,k_1},c_{k_1k_2},\cdots,c_{k_{n-1}k_n},c_{k_n,j}$ are nonzero numbers. In this subsection, we will prove the following conclusion.

\begin{proposition}
Suppose that $\bigtriangleup$ is a finite root system of Cartan type, then there exist an abelian group $\G$ and a root datum
$\mathfrak{D}=\mathfrak{D}(\mathcal{D}_{\chi,E},\mathbf{S},\mathbf{X})$ over $\G,$ such that $\MM(\mathfrak{D})$ is a finite-dimensional genuine pointed Majid algebra and the root system of $\mathfrak{D}$ is $\bigtriangleup.$
\end{proposition}

\begin{proof}
Let $\mathrm{C}=(c_{ij})_{1\leq i,j\leq \theta}$ be the finite Cartan matrix corresponding to $\bigtriangleup,$ $\mathrm{D}=\d(d_1,\cdots,d_\theta)$ the diagonal matrix such that $\mathrm{DC}$ is symmetric. Let $\mathfrak{J}$ be the set of connected components of $\{1,\cdots,\theta\}.$ Fix an order $<$ on $\mathfrak{J}$.
For each $I\in \mathfrak{J}$, define a positive odd integer $\m_I>2$ satisfying
\begin{itemize}
\item[T1.]If $I,I'\in \mathfrak{J}$ and $I<I'$, then $\m_I|\m_{I'}$;
\item[T2.]If $I$ is of type $G_3$, we assume that $3\nmid \m_I$.
\end{itemize}
Let $q_I=\zeta_{\m_I^2}$ be a primitive $\m_I^2$-th  root of unity, and $q_{ii}=q_I^{d_i},i\in I, I\in \mathfrak{J}.$ For all $1\leq i<j\leq \theta,$ define $q_{ij}=q_{ii}^{-d_ic_{ij}}, q_{ji}=1.$ Let $E=\{e_1,\cdots,e_\theta\}$ be the canonical basis of $\Z^\theta$ and $\chi$ a bicharacter on $\Z^\theta$ given by $$\chi(e_i,e_j)=q_{ij},\quad 1\leq i,j\leq \theta.$$
Then it is obvious that $\triangle_{\chi,E}$ is an arithmetic root system. Set $\G=\Z_{\m_1}\times \cdots\times \Z_{\m_\theta}=\langle \g_1\rangle\times\cdots\times \langle \g_\theta\rangle,$ here $\m_i=\m_I,i\in I.$ Next we will show that there exists a root datum $\mathfrak{D}=\mathfrak{D}(\mathcal{D}_{\chi,E},\mathbf{S},\mathbf{X})$ over $\G.$

Let $\mathbf{S}=(s_{ij})_{1\leq i,j\leq \theta}$ be the identity matrix, i.e., $s_{ij}=\delta_{ij},1\leq i,j\leq \theta,$ then the inverse matrix $T=(t_{ij})$ is also identity matrix. For all $1\leq i,j \leq \theta$, define $\mathbf{X}=(x_{ij})_{1\leq i,j\leq \theta}$ through
\begin{equation*}
x_{ij}=\left\{
         \begin{array}{ll}
           d_i, & \hbox{$i=j$;} \\
           -d_ic_{ij}, & \hbox{$i<j$;} \\
           0, & \hbox{$i>j$.}
         \end{array}
       \right.
\end{equation*}
Then we have
\begin{equation*}
\sum_{j=1}^m x_{ij}t_{kj}=x_{ik}\equiv 0 \; (\textrm{mod}\; \mathbbm{m}_i), \; 1\leq k<i\leq n,
\end{equation*}
which implies \eqref{e4.7}.

According to the definition of $x_{ij}$, it is obvious that $\prod_{k=1}^{n}\zeta_{m_k}^{s_{ik}x_{ki}}=\zeta_{m_i}^{x_{ii}}=q_{ii},$ here $m_i=\m_i^2$ for $1\leq i\leq \theta$. When $i<j$, we have  $\prod_{k=1}^{n}\zeta_{m_k}^{s_{ik}x_{kj}+s_{jk}x_{ki}}=q_{ij}^{-x_{ij}}=q_{ij}q_{ji}=\widetilde{q_{ij}}$. This implies $\mathfrak{D}=\mathfrak{D}(\mathcal{D}_{\chi,E},\mathbf{S},\mathbf{X})$ is a root datum over $\G$. Since $a_i\equiv x_{ii}=d_i\;(\textrm{mod} \;\m_i)$,  $a_i\neq 0$ for each $i$. Hence $\underline{\textbf{a}}$ is nontrivial, which implies that $\MM(\mathfrak{D})$ is genuine.

Finally we prove that $\MM(\mathfrak{D})$ is finite-dimensional. We need to show that for any $\alpha\in \bigtriangleup^+$, the nilpotent index $N_\alpha$ is finite. Let $\bigtriangleup_I$ be the root system corresponding to $I\in\mathfrak{J}$, it is obvious that $\bigtriangleup=\bigcup_{I\in\mathfrak{J}}\bigtriangleup_I$. Let $G$ be the bigger group define by $\G$ and $\pi,\iota,$ see the sentences before Lemma \ref{l4.4}. We know that there exists a $2$-cocycle $J$ on $G$, such that $U=\widetilde{V(\mathfrak{D})}^{J^{-1}}\in{ _{\k G}^{\k G}\mathcal{YD}}$ and $\B(U)$ and $\B(\mathfrak{D})$ have the same root system. Because the nilpotent index of a root vector is invariant under twisted deformation by $J$, we only need to prove the nilpotent index of the arithmetic root vector of $\B(U)$ is finite. According to \cite[Theorem 5.1]{AS1} and (T2), for all $\alpha\in \bigtriangleup^+_I,I\in \mathfrak{J}$£¬$N_\alpha=N_I=|q_{ii}|,i\in I,$ hence must be finite.

We complete the proof of the proposition.
\end{proof}

\begin{remark}
From the proof of the proposition, we see that there are many choices of $\G$, hence many pointed Majid algebra associated to each arithmetic root system of Cartan type. This also provides a large class of examples of new genuine finite-dimensional pointed Majid algebras.
\end{remark}

\end{document}